\documentclass[preprint,12pt]{elsarticle}

\usepackage{graphicx}
\usepackage{epsfig}

\usepackage{amssymb}
\usepackage{amsthm}
\usepackage{amsmath,amsfonts,amssymb,amsthm}
\usepackage{latexsym}
\usepackage{epsf}
\usepackage{epstopdf}
\usepackage{psfrag}
\usepackage{indentfirst}
\usepackage{eucal}
\usepackage{fontenc}
\usepackage{float}
\usepackage[all]{xy}

\def\ve{\varepsilon}
\def\al{\alpha}
\def\A{\mathcal{A}}

\def\la{\lambda}

\def\f{\tilde{f}}

\def\d{\mathrm{d}}
\def\O{\mathrm{O}}

\def\Z{\mathcal{Z}}
\newcommand{\dem}{\begin{proof}}
\newcommand{\cqd}{\end{proof}}

\theoremstyle{plain}
\newtheorem{lemma}{\bf Lemma}[section]

\newtheorem{cor}[lemma]{Corollary}

\newtheorem{pro}[lemma]{Proposition}
\newtheorem{maintheorem}{Theorem}

\theoremstyle{definition}
\newtheorem{definicao}[lemma]{Definition}

\journal{Analyse Non Linéaire}

\begin{document}

\begin{frontmatter}



\title{Renormalization for piecewise smooth homeomorphisms on the circle}


%\author[kmc,ds]{Kleyber Cunha and Daniel Smania}
%\ead[kmc]{kleyber@ufba.br, smania@icmc.usp.br}
%\ead[url]{www.icmc.usp.br/$\sim$smania/}

\author[kmc]{Kleyber Cunha\corref{cor1}\fnref{kmcc}}
\ead{kleyber@ufba.br}
\ead[url]{www.sd.mat.ufba.br/$\sim$kleyber/}

\author[ds]{Daniel Smania\fnref{dss}}
\ead{smania@icmc.usp.br}
\ead[url]{www.icmc.usp.br/$\sim$smania/}

\cortext[cor1]{Correspondig author}

\address[kmc]{Departamento de Matem\'atica,
Universidade Federal da Bahia, Av. Ademar de Barros s/n,
CEP 40170-110
Salvador-BA, Brazil}

\address[ds]{Departamento de Matem\'atica,
ICMC/USP-Campus de S\~ao Carlos, Caixa Postal 668,  CEP 13560-970, S\~ao Carlos, SP, Brazil}

\fntext[kmcc]{K.C. was partially supported by FAPESP 07/01045-7. }

\fntext[dss]{D.S. was partially supported by
FAPESP 2008/02841-4 and 2010/08654-1, CNPq 310964/2006-7 and 303669/2009-8.}\begin{abstract}
In this work we study the renormalization operator acting on piecewise smooth homeomorphisms on the circle, that turns out to be essentially the study of Rauzy-Veech renormalizations of generalized interval exchanges maps  with genus one. In particular we show that renormalizations of such maps with  zero mean nonlinearity and satisfying certain smoothness and combinatorial assumptions converges to the set of  piecewise affine interval exchange maps.
\end{abstract}

%\begin{abstract}
%Dans ce travail nous étudions l'opérateur de renormalisation agissant sur des homéomorphismes du cercle lisses par morceaux, cela va essentiellement se traduire par l'étude des renormalisations de Rauzy-Veech d'échanges d'intervalles de genre un. En particulier nous montrons que les renormalisations de telles applications avec une non-linéarité de moyenne nulle et satisfaisant certaines hypothèses de régularités et combinatoires convergent vers l'ensemble des échanges d'intervalles affines par morceaux.
%\end{abstract}


\end{frontmatter}

% \linenumbers

\section{Introduction and results}
\label{}
One of the most studied classe of dynamical systems are  orientation-preserving diffeomorphisms of the
circle.  It may be classified according to their rotation number $\rho(f)$ which, roughly speaking, measures the average rate of rotation of orbits around the circle. When $\rho(f)\in\mathbb{Q}$ then $f$ has a periodic point and all other orbits will converge to some periodic orbit both in the future and in the past. If $\rho(f)$ is irrational then $f$ has not periodic point and its dynamics depends on the smoothness of $f.$ Denjoy result ensures that if $f$ is $C^2$ then it is conjugate to the rigid rotation of angle $\rho(f).$ In this context, it is a natural question to ask  under what conditions the conjugacy is smooth. Several authors, Herman \cite{herman}, Yoccoz \cite{yoccoz}, Khanin and Sinai \cite{newproof, khanin1}, Katznelson and Ornstein \cite{katz}, have shown that if $f$ is $C^3$ or $C^{2+\nu}$ and $\rho(f)$ satisfies certain diophantine condition  then the conjugacy will be at least $C^1$.

A natural generalization of diffeomorphisms of the circle are diffeomorphisms with breaks, i.e., $f$ has  jumps in the first derivative  on finitely many points.  In this setting Khanin and Vul \cite{khanin} show that for diffeomorphisms with one break the renormalization operator converges to a two dimensional space of the fractional linear transformation. Our first main result generalizes  results of Khanin and Vul \cite{khanin} for finitely many break points. A key combinatorial method in our proof is to consider a piecewise smooth homeomorphism on the circle as a generalized interval exchange transformation. %One of the tools most used in one-dimensional dynamic is a renormalization group technique. For example see \cite{khanin} for circle homeomorphism and unimodal maps. For i.e.m. there is the Rauzy-Veech induction operator which assigns to each i.e.m. its first return map to a convenient subinterval (\cite{rauzy}, \cite{veech1}). More precisely, in this paper, we are interested in asymptotic behavior of the Rauzy-Veech renormalization operator, defined by composing the induction operator with a rescaling of the domain.

Let $I$ be an interval  and let  $\A$ be a finite set (the {\it  alphabet}) with $d~\geq~2$ elements and $\mathcal{P} =\{I_\al: \al\in\A\}$ be an $\mathcal{A}-${\it indexed} partition of $I$ into subintervals \footnote{All the subintervals will be bounded, close on the left and open on the right.}. We say that the triple $(f,\mathcal{A}, \mathcal{P})$, where  $f:I\to I$ is a bijection, is a generalized interval exchange transformation with $d$ intervals  (g.i.e.m. with $d$ intervals,  for short), if $f|_{I_\al}$ is an orientation-preserving homeomorphism for each $\alpha \in \mathcal{A}$. Most of the time we will abuse the notation saying that $f$ is a g.i.e.m. with $d$ intervals. The order of the subintervals in the domain and image constitue the combinatorial data for $f,$ where explicitly defined in the next section.

When $f|_{I_\al}$ is a translation we say which $f$ is a standard i.e.m.
 Standard i.e.m. arise naturally as Poincaré return maps of measured foliations and geodesic flows on translation surfaces. But they are also interesting examples of simple dynamical systems with very rich dynamics and have been extensively studied for their own sake. When $d=2,$ by identifying the endpoints of $I,$ standard i.e.m. correspond to rotations of the circle and generalized i.e.m. correspond to circle homeomorphisms.

 In another article, Khanin and Sinai \cite{newproof} show a new proof of M. Herman's theorem. From the viewpoint of the renormalization group approach they show the  convergence of the renormalizations of a circle diffeomorphism to the linear fixed point of the renormalization operator  for  diffeomorphisms of the circle.  We use a similar approach to study  generalized interval exchange maps of genus one.

%The order of the subintervals in the domain and image constitue the combinatorial data for the i.e.m. $f.$

% Given two g.i.e.m, one can ask if there are conjugate. There are combinatorial restrictions for the existence of this conjugacy, as we are going to see below.   Suppose that they are conjugate.
%We are mainly interested in how regular is this conjugacy. Indeed our main result claims that if the g.i.e.m. satisfies certain combinatorial and metrical assumptions then the conjugacy is at least $C^1$.

\subsection{Renormalization: Rauzy-Veech induction} \label{renrv} To describe the combinatorial assumptions of our results, we need to introduce the Rauzy-Veech scheme. This is a renormalization scheme. Renormalization group techniques are a very powerful tool in one-dimensional dynamics. For example see Khanin and Vul \cite{khanin} for circle homeomorphisms and de Melo and van Strien \cite{melo} for unimodal maps.

Following the algorithm of Rauzy \cite{rauzy} and Veech \cite{veech}, for every i.e.m. $f$ without connections, we define the Rauzy-Veech induction by considering the first return maps $f_n$ of $f$ on a decreasing sequence of intervals $I^n,$ with the same left endpoint than $I.$ The map $f_n$ is again generalized i.e.m. with the same alphabet $\A$ but the combinatorial data may be different.

 Given two intervals $J$ and $U$, we will write $J < U$ if their interior are disjoint and $x < y$ for every $x \in J$ and $y \in U$. This defines a partial order in the set of all intervals. Denote the lenght of an interval $J$ by $|J|$.

  Given a g.i.e.m. $f\colon I^0\rightarrow I^0$, with alphabet $\mathcal{A}$. Let $\pi_j\colon \mathcal{A}\rightarrow \{1,\dots,d\}$, with $j=0,1$, be bijections  such that

 $$I_{\alpha}< I_{\beta}$$
 iff $\pi_0(\alpha) < \pi_0(\beta)$ and
 $$f(I_{\alpha}) < f(I_{\beta})$$
 iff $\pi_1(\alpha) < \pi_1(\beta)$.

 The pair $\pi=\pi(f)=(\pi_0,\pi_1)$ is called the combinatorial data associated to the g.i.e.m. $f$. We call
 \begin{eqnarray}
 \label{mono}
 p=\pi_1\circ \pi_0^{-1}:\{1,\ldots,d\}\to \{1,\ldots,d\}
 \end{eqnarray}
 the {\it monodromy invariant} of the pair $\pi=(\pi_0,\pi_1).$ When appropriate we will use the notation $p=(p(1)\ p(2)\ ... \ p(d))$ for the data combinatorial of $f.$ We also assume that the pair $\pi=(\pi_0,\pi_1)$ is irreducible, i.e., for all $j\in\{1,...,d-1\}$ we have $\pi_1\circ\pi_0^{-1}(\{1,...,j\})\neq \{1,...,j\}.$

 For each $\ve\in\{0,1\}$, define $\al(\ve)=\pi^{-1}_{\ve}(d)$. If $|I_{\alpha(0)}| \neq |f(I_{\alpha(1)})|$ we say that $f$ is Rauzy-Veech renormalizable (or simply renormalizable, from now on).  If $|I_{\alpha(0)}| > |f(I_{\alpha(1)})|$ we say that the letter $\alpha(0)$ is the winner and the letter  $\alpha(1)$ is the loser. We say that $f$ is   type $0$ renormalizable and we can define a map $\hat{R}(f)$ as the first return map of $f$  to the interval
 $$I^1= I \setminus f(I_{\alpha(1)}).$$ Otherwise $|I_{\alpha(0)}| < |f(I_{\alpha(1)})|$,  the  letter $\alpha(1)$ is the winner and the letter  $\alpha(0)$ is the loser, we say that $f$ is type $1$ renormalizable and we can define a map $\hat{R}(f)$ as the first return map of $f$  to the interval
 $$I^1= I \setminus I_{\alpha(0)}.$$

 We want to see $R(f)$ as a g.i.e.m. To this end we need to associate to this map a $\mathcal{A}$-indexed partition of its domain. We do this in the following way. The  subintervals of the $\mathcal{A}-$ partition $\mathcal{P}^1$ of $I^1$ are  denoted by $I^1_{\al}.$ If  $f$ has type 0, $I^1_{\al}=I_{\al}$. If $\al\not=\al(0)$, $I^1_{\al(0)}=I_{\al(0)}\setminus f(I_{\al(1)})$ and when $f$ has type 1, $I^1_{\al}=I_{\al}$ if $\al\not=\al(1),\al(0)$, $I^1_{\al(1)}=f^{-1}(f(I_{\al(1)})\setminus I_{\al(0)})$ and $I^1_{\al(0)}=I_{\al(1)}\setminus I^1_{\al(1)}.$ It is easy to see that in both cases (type 0 and 1) we have

$$
R(f)(x)=\left\{
\begin{array}{ll}
f^2(x),& \mbox{~se~} x\in I^1_{\al(1-\ve)},\\
f(x), & \mbox{~otherwise}.
\end{array}
\right.
$$
and $(R(f), \mathcal{A}, \mathcal{P}^1)$ is a g.i.e.m., called the {\bf Rauzy-Veech renormalization} (or simply renormalization, for short)  of $f$.  If $f$ is renormalizable with type $\ve \in \{0,1\}$ then the combinatorial data $\pi^1=(\pi^1_0,\pi^1_1)$ of $R(f)$ is given by

$$
\pi^1_{\ve}:=\pi_{\ve}$$
and
$$\mbox{~and~} \pi^1_{1-\ve}(\al)=\left\{
\begin{array}{ll}
\pi_{1-\ve}(\al),& \mbox{~se~}\pi_{1-\ve}(\al)\leq\pi_{1-\ve}(\al(\ve)),\\
\pi_{1-\ve}(\al)+1,& \mbox{~se~}\pi_{1-\ve}(\al(\ve))<\pi_{1-\ve}(\al)<d,\\
\pi_{1-\ve}(\al(\ve))+1,& \mbox{~se~}\pi_{1-\ve}(\al)=d.
\end{array}
\right.
$$
Since $\pi^1$ depends only on $\pi$ and the type $\ve$, we denote  $r_\ve(\pi)=\pi^1$.

A g.ie.m. is infinitely renormalizable if $R^n(f)$ is well defined, for every
$n \in \mathbb{N}$. For every interval of the form $J=[a,b)$ we denote $\partial J=\{a\}.$ We say that a g.i.e.m $f$ has no connection if
$$f^m(\partial I_\alpha)\neq \partial I_\beta \text{ for all } m\geq1,\; \alpha,\beta\in\mathcal{A} \text{ with } \pi_0(\beta)\neq1.$$
This property is invariant under iteration of $R.$ Keane \cite{keane} show that no connection condition is a necessary condition for $f$ to be infinitely renormalizable.

 Let $\ve_n$ the type of the $n-$th renormalization,  $\alpha_n(\ve_n)$ be the winner and $\alpha_n(1-\ve_n)$ be the loser of the $n-$th renormalization.

 We say that infinitely renormalizable g.i.e.m. $f$ has {\bf  $k-$bounded combinatorics}
if for each $n$ and $\beta, \gamma  \in \mathcal{A}$ there exists $n_1, p \geq 0$, with $|n-n_1|< k$ and $|n-n_1-p|< k$, such that
$\alpha_{n_1}(\ve_{n_1})=\beta$, $\alpha_{n_1+p}(1-\ve_{n_1+p})=\gamma$ and $$\alpha_{n_1+i}(1-\ve_{n_1+p})= \alpha_{n_1+i+1}(\ve_{n_1+i})$$
for every $0\leq i < p$.

We say that a g.i.e.m. $f\colon I \rightarrow I$  has {\bf genus one} by Veech \cite{veech1}  (or belongs to the {\bf rotation class} by Nogueira and Rudolph \cite{nogueira})  if $f$ has at most two discontinuities. Note that  every g.i.e.m. with either two or three intervals has genus one. If $f$ is renormalizable and has genus one, it is easy to see that $R(f)$ has genus one.

Given two infinitely renormalizable g.i.e.m. $f$ and $g$, defined with the same alphabet $\mathcal{A}$, we say that $f$ and $g$ have the {\bf same combinatorics} if $\pi(f)=\pi(g)$ and the the $n-$th renormalization of $f$ and $g$ have the same type, for every $n \in \mathbb{N}$. It follows that if $\pi^n(f)=\pi^n(g)$ for every $n$, where $\pi^n(f)$ is the combinatorial data of the $n-$th renormalization of $f$.

\begin{definicao} Let $\mathcal{B}_k^{2+\nu}$, $k \in \mathbb{N}$ and  $\nu > 0$,  be the set of g.i.e.m.  $f:I\to I$ such that
\begin{itemize}
\item[(i)]  For each $\alpha \in \mathcal{A}$ we can extend $f$ to $\overline{I_\al}$ as an orientation-preserving diffeomorphism of class $C^{2+\nu}$;\\
\item[(ii)] the g.i.e.m. $f$  has  $k-$bounded combinatorics;\\
\item[(iii)] The map   $f$ has  genus one and has no connection;
\end{itemize}
\end{definicao}

%For circle homeomorphism the renormalization group technique has been studied in recent years. For example, in \cite{khanin} the authors show that the renormalizations of  circle homeomorphisms with one weak discontinuites and satisfying certain combinatorial restrictions  are asymptotically piecewise  fractional-linear functions.

Let $H$ be  a non-degenerate interval, let  $g:H\to\mathbb{R}$ be a diffeomorphism  and let $J\subset H$ be an interval. We define the {\it Zoom} of $g$ in $H,$  denoted by $\Z_H(g),$ the transformation $\Z_H(g)=A_1\circ g\circ A_2,$ where $A_1$ and $A_2$ are orientation-preserving affine maps, which send $[0,1]$ into $H$ and $g(H)$ into $[0,1]$ respectively. Consider the set $C^2([0,1],\mathbb{R})$  of all $C^2$ functions $g\colon [0,1]\rightarrow \mathbb{R}$ with the usual norm
$$d_{C^2}(f,g):= \sum_{i=0}^2 \sup_{x \in [0,1]} |D^{(i)}f(x)-D^{(i)}g(x)|,$$
where $D^{(i)}f$ and $D^{(i)}g$ denote the i-th derivative of $f$ and $g$ respectively.

Denote by $\mathcal{M}$ the set of M\"oebius transfomations $M:[0,1]\to[0,1]$ such that $M(0)=0$ and $M(1)=1.$
Note that $\mathcal{M}$ is an one-dimensional real Lie group. Indeed any element $M \in \mathcal{M}$ has the form

\begin{eqnarray}
\label{defMoebius}
M=M_{N}(x)=\dfrac{xe^{\frac{-N}{2}}}{1+x(e^{\frac{-N}{2}}-1)}.
\end{eqnarray}
for some $N \in \mathbb{R}$ and  $M_{N_1}\circ M_{N_2}=M_{N_1+N_2}$.
Moreover $M_{N}$ is the unique M\"oebius transformation $M$ which sends $[0,1]$ onto $[0,1],$ $M(0)=0,$ $M(1)=1,$ and
$$\int_0^1 \frac{D^2M(x)}{DM(x)} \ dx= N.$$

\subsection{Main results}

\begin{maintheorem}
\label{teo1}
Let $f\in\mathcal{B}_k^{2+\nu}.$ Then there are $C=C(f)>0$ and  $0<\la=\la(k)<1$ such that
$$\mathrm{d}_{C^2}(\mathcal{Z}_{I^n_{\al}}(R^n(f)),M_{N^n_\alpha})\leq C \la^n$$
for all $\al\in\A.$ Here
$$N^n_\alpha = \int_{I^n_\alpha} \frac{D^2R^n(f)(x)}{DR^n(f)(x)} \ dx.$$
In particular
$$\mathrm{d}_{C^2}(\mathcal{Z}_{I^n_{\al}}(R^n(f)),\mathcal{M})\leq C \la^n.$$
\end{maintheorem}

We can say more about the mean nonlinearities  $N^n_\alpha$. Denote by $q^n_\alpha\in\mathbb{N}$ the first return time of the interval $I^n_\alpha$ to the interval $I^n,$ i.e., $\hat{R}^n(f)|_{I^n_\alpha}=f^{q^n_\alpha}$, for some $q^n_\alpha \in \mathbb{N}$.

\begin{maintheorem}
\label{teon4}\label{principal} Let $f\in\mathcal{B}_k^{2+\nu}.$ Then there are $C=C(f)>0$ and  $0<\la=\la(k)<1$ such that
\begin{equation} \label{rat3} |N^n_\alpha - \frac{\sum_{i=1}^{q^n_\alpha} |f^i(I^n_\alpha)|}{|I|} \int \frac{D^2f(x)}{Df(x)} dx|\leq C\lambda_{\ }^{\sqrt{n}}.\end{equation}
In particular if
$$\int_{[0,1]} \frac{D^2f(x)}{Df(x)} dx=0$$
then $|N^n_\alpha|< C\lambda_{\ }^{\sqrt{n}}$.
\end{maintheorem}

The rate of convergence obtained in (\ref{rat3}) is enough for our purposes.    In  the case of circle diffeomorphisms  Khanin and Teplinsky \cite{rev} obtained an exponential rate using a different  approach.

%Here $\mathcal{Z}_{I^n_{\al}}(R^n(f))$ denote the zoom of $R^n(f)$ on $I^n_{\al},$ this is $\mathcal{Z}_{I^n_{\al}}(R^n(f))=A_1\circ R^n(f)\circ A_2,$ where $A_1$ and $A_2$ be affine maps which send $[0,1]$ into $I^n_{\al}$ and $R^n(f)(I^n_{\al})$ into $[0,1].$

Our third result is an almost direct consequence of Theorem 1 and 2.

\begin{maintheorem}
\label{teo2}
Let $f\in\mathcal{B}_k^{2+\nu}$ such that
$$\int_{[0,1]} \frac{D^2f(x)}{Df(x)} dx=0.$$
Then there are $C=C(k)>0$ and $0<\la=\la(k)<1$ such that
$$ \left|\mathcal{Z}_{I^n_{\al}}(R^n(f))-\mathrm{Id}\right|_{C^2}\leq C\cdot \la^{\sqrt{n}}  \mbox{~~for all~~} \al\in\A.$$
\end{maintheorem}

%From the viewpoint of the renormalization group ideology the result proved means the convegence to the affine i.e.m..

The structure of this paper is as follows.  In Section \textsection \ref{comsec} we describe general results on compositions of diffeomorfisms of class $C^{2+\nu}$. In Section \textsection \ref{rensec} we study renormalization of generalized interval exchange maps of genus one and prove Theorem \ref{teo1}. In  Section \textsection \ref{synsec} we codify the dynamics of $f$ using  an specially crafted symbolic dynamics to obtain finer geometric properties of the partitions associated with renormalizations of $f$ and we finally prove Theorem \ref{teon4} and Theorem \ref{teo2}.

This is the first of a series of two papers based on the Ph. D. Thesis of the first author Cunha \cite{tese}. In the second work \cite{artigo2} we continue our study of  the renormalization operator for generalized interval exchange transformations of genus one and its consequences, particularly  the rigidity (universality) phenomena in the setting of piecewise smooth homeomorphisms on the circle.

\section{Comparing compositions of  $C^{2+\nu}$ maps with Moebius maps }\label{comsec}

In this section, we show some results about composition of  $C^{2+\nu}$ diffeomorphisms, comparing these compositions with Moebius maps.  Let $f:[a,b]\to [f(a),f(b)]$  be an $C^2$ orientation-preserving diffeomorphism.  Define the {\it nonlinearity function} $n_f:[a,b]\to\mathbb{R}$ by
$$n_f(x)=\dfrac{D^2f(x)}{Df(x)}=D(\ln Df(x)).$$
Notice that
$$n_{f\circ g}(x) = n_f(g(x))Dg(x) + n_f(x),$$
consequently  if $f_i$ are $C^2$ diffeomorphisms  such that $f=f_n\circ \cdots \circ f_1$ is defined in $[a,b]$ we have
\begin{equation}\label{noncomp} \int_{[a,b]} \frac{D^2f(x)}{Df(x)} \ dx = \sum_{i=1}^n \int_{f^{i-1}[a,b]} \frac{D^2f_i(x)}{Df_i(x)}  \ dx.  \end{equation}
If $[a,b]=[0,1]$ we define
$$N_f =\int_{[0,1]} \dfrac{D^2f(x)}{Df(x)} \ dx.$$

The nonlinearity $n_f$ defines $f$ up its domain and image. Indeed,   given a continuous  function $n\colon [0,1]\rightarrow \mathbb{R}$ there is unique $C^{2}$-diffeomorphism $f\colon [0,1]\rightarrow [0,1]$ such that $f(0)=0$, $f(1)=1$ and $n_f=n$. Indeed, see Martens \cite{martens}
\begin{eqnarray}
\label{recuperada}
f(x)=\dfrac{\int_0^x\exp \left(\int_0^{z} n(y)dy\right)dz}{\int_0^1\exp \left(\int_0^{z} n(y)dy\right)dz}.
\end{eqnarray}
%Define the {\it zoom} of $f$ on $[a,b]$ by $Z_{[a,b]}(f) = Z(f)= A_2\circ f \circ A_1,$ where $A_1$ and $A_2$ be affine maps which send $[0,1]$ in $[a,b]$ and $[f(a),f(b)]$ in $[0,1],$ respectively. So the nonlinearity of $Z(f)$ is given by
Let $f\colon [0,1] \rightarrow [f(0),f(1)]$ be an $C^2$ orientation-preserving  diffeomorphism.  If $[a,b] \subset [0,1]$, let $\tilde{f}= \mathcal{Z}_{[a,b]}(f)$  be the Zoom of $f$ in $[a,b].$ Then
\begin{eqnarray}
\label{nlZ}
n_{\Z(f)}(x)=(b-a) \cdot n_f(a+x(b-a)).
\end{eqnarray}
Suppose
\begin{eqnarray}
\label{nlh}
|n_f(x)-n_f(y)|\leq  C_0\cdot |x-y|^{\nu}.
\end{eqnarray}
for $x,y \in [0,1]$ and
$$|n_f|_{C^0[0,1]}:=\sup_{x\in[0,1]}\{|n_f(x)|\}\leq C_1.$$

Then by \eqref{nlh} and \eqref{nlZ} we have that
\begin{eqnarray}
\label{onl}
|n_{\Z(f)}(x)-n_{\Z(f)}(y)| \leq C_0 \cdot \delta^{1+\nu} \mbox{~~and~~~} |n_{\Z(f)}|_{C^0[0,1]}\leq C_1\cdot\delta,
\end{eqnarray}
with $x,y \in [0,1]$ and  $\delta=b-a.$ Note that

\begin{equation}\label{zoomnon} N_{Z(f)}= \int_0^1 (b-a) \cdot n_f(a+x(b-a)) \ dx = \int_{[a,b]} \frac{D^2f(x)}{Df(x)} \ dx.\end{equation}

%The following proposition gives a relationship between $n_f$ and $n_{\Z(f)}.$

%\begin{pro}
%\label{nz}
%Let $f:[a,b]\to [f(a),f(b)]$  be an orientation-preserving diffeomorphism of class $C^2$ and $\mathcal{Z}_{[a,b]}(f)=\mathcal{Z}(f):[0,1]\to[0,1]$ the zoom of $f$ in $[a,b].$ Then
%$$|n_{\Z(f)}(x)-n_{\Z(f)}(y)| \leq C \cdot \delta^{1+\nu} \mbox{~~and~~~} |n_{\Z(f)}|_{C^0}\leq \delta\cdot |n_{f}|_{C^0},$$
%where $\delta=b-a$ and $|\cdot|_{C^0}$ denote the norm of the $C^0-$topology.
%\end{pro}

%\dem Follows directly from the fact that
%\begin{eqnarray*}
%n_{\Z(f)}(x)=(b-a) \cdot n_f(a+x(b-a)).
%\end{eqnarray*}

%\cqd

\begin{pro}
\label{moebius}
Let $f:[0,1]\to [f(0),f(1)]$  be an orientation-preserving diffeomorphism of class $C^{2+\nu},$  $[a,b]\subset [0,1]$ and define $\tilde{f}=Z_{[a,b]}f$.  Then
$$\d_{C^2}(\f,M_{N_{\f}})=\O(\delta^{1+\nu}),$$
where $\delta =b-a$.
\end{pro}

Before we prove the Proposition \ref{moebius} we prove the following lemma:

\begin{lemma}
\label{ncte} Let $N \in \mathbb{R}$.
Let $f_N\colon [0,1]\mapsto [0,1]$ be a diffeomorphism such that $n_f(x)=N$ for all $x\in [0,1]$, $f(0)=0$, $f(1)=1$. Then
$$\d_{C^2}(f_N,M_{N})=\O(N^2).$$
\end{lemma}

\dem By \eqref{recuperada} we have
$$f_{N}(x)=\dfrac{\int_0^x e^{Nz}dz}{\int_0^1e^{Nz}dz}=\dfrac{e^{Nx}-1}{e^N-1}.$$
Therefore,

\begin{eqnarray*}
|f_N(x) - M_{N}(x)|& = & \left|\dfrac{e^{Nx}-1}{e^N-1}-\dfrac{xe^{\frac{-N}{2}}}{1+x(e^{\frac{-N}{2}}-1)}\right|\\
& = & \left| \dfrac{Nx+\frac{N^2x^2}{2}+\O(N^3)}{N+\frac{N^2}{2}+\O(N^3)}-\dfrac{x(1-\frac{N}{2}+\O(N^2))}{1+x(-\frac{N}{2}+\O(N^2))}\right|\\
%& = &  \left|  \dfrac{Nx(1+\frac{Nx}{2}+\O(N^2))}{N(1+\frac{N}{2}+\O(N^2))}- x(1-\frac{N}{2}+\O(N^2))(1+\frac{Nx}{2}+\O(N^2))\right|\vspace{0.5cm}\\
& = &  \left| x(1+\dfrac{Nx}{2}-\dfrac{N}{2})+\O(N^2)- x(1+\dfrac{Nx}{2}-\dfrac{N}{2})+\O(N^2)\right|\\
&= & \O(N^2).
\end{eqnarray*}

\begin{eqnarray*}
|Df_N(x) - DM_{N}(x)|& = & \left|\dfrac{Ne^{Nx}}{e^N-1}-\dfrac{e^{\frac{-N}{2}}}{[1+x(e^{\frac{-N}{2}}-1)]^2}\right|\\
& = &  \left| \dfrac{(1+Nx+\O(N^2))}{(1+\frac{N}{2}+\O(N^2))}-  e^{-\frac{N}{2}}(1+\dfrac{Nx}{2}+\O(N^2))^2\right|\\
& = &  \left|  (1+Nx)(1-\dfrac{N}{2})+\O(N^2) -  e^{-\frac{N}{2}}\left(1+Nx+\O(N^2)\right)\right|\\
& = &  \left| 1-\dfrac{N}{2}+Nx+\O(N^2)- 1-\dfrac{N}{2}+Nx+\O(N^2)\right|\\
&= & \O(N^2).
\end{eqnarray*}

\begin{eqnarray*}
%\label{seg_der}
|D^2f_N(x) - D^2M_{N}(x)|& =& \left|\dfrac{N^2e^{Nx}}{e^N-1}-\dfrac{-2e^{\frac{-N}{2}}(e^{-\frac{N}{2}}-1)}{[1+x(e^{\frac{-N}{2}}-1)]^3}\right|\\
& = &  \left| N+\O(N^2) - (N+\O(N^2))(1+\dfrac{Nx}{2}+\O(N^2))^3 \right|\\
&= & \O(N^2).
\end{eqnarray*}
\cqd

\begin{proof}[Proof of Proposition \ref{moebius}.] By Eq. (\ref{onl}) we have
$$|N_{\tilde{f}} |\leq |n_f|_{C^0[0,1]} \delta.$$
To simplify the notation, denote $\tilde{N}=N_{\tilde{f}}$. First note that
$$\d_{C^2}(\f,M_{\tilde{N}})\leq \d_{C^2}(\f,f_{\tilde{N}})+\d_{C^2}(f_{\tilde{N}},M_{\tilde{N}}).$$
In view of Lemma \eqref{ncte}, only need to estimate the first term of the right hand side. For this note that $\tilde{N}=\int_0^1n_{\f}(s)ds=n_{\f}(\theta)$ for some $\theta\in[0,1].$ If
$$|n_f(x)-n_f(y)|\leq C_0 |x-y|^{\nu},$$
then by Eq. (\ref{onl}) we have
 $|n_{\f}(x)-\tilde{N}|=|n_{\f}(x)-n_{\f}(\theta)|\leq C_0\cdot\delta^{1+\nu},$ so  $n_{\f}(x)=\tilde{N}+\O(\delta^{1+\nu}).$
Then by \eqref{recuperada} we obtain

\begin{eqnarray}
\f(x)&=&x(1-\dfrac{\tilde{N}}{2}+\dfrac{\tilde{N}}{2}x+\O(\delta^{1+\nu}));\label{estZ}\\
D\f(x)&=&1+\tilde{N}x-\dfrac{\tilde{N}}{2}+\O(\delta^{1+\nu});\label{estDZ}\\
D^2\f(x)&=&\tilde{N}+\O(\delta^{1+\nu}).\label{estD2Z}
\end{eqnarray}

Using the estimates  for $f_{\tilde{N}},$ $Df_{\tilde{N}}$ and $D^2 f_{\tilde{N}}$ similar to those in the proof  of  Lemma \ref{ncte}, the proof is complete. \end{proof}

From now on let $f_i: [0,1]\to  [0,1]$, $i \in \mathbb{N}$ be  orientation-preserving diffeomorphisms of class $C^{2+\nu}$, with $f_i(0)=0$, $f_i(1)=1$, and such that there exist $C_0, \ C_1 > 0$ satisfying
\begin{eqnarray}
\label{nlh2}
|n_{f_i}(x)-n_{f_i}(y)|\leq  C_0\cdot |x-y|^{\nu}.
\end{eqnarray}
and
$$|n_{f_i}|_{C^0[0,1]}\leq  C_1.$$
for every $i \in \mathbb{N}$. Let $[a_i,b_i]\subset [0,1]$, $\delta_i =b_i-a_i$, and    $\f_i=\mathcal{Z}_{[a_i,b_i]}(f_i),$ $M_i=M_{N_{\f_i}},$
$$\f_1^n=\f_n\circ \f_{n-1}\circ \cdots\circ\f_1 \mbox{~~and~~}M_1^n=M_n\circ M_{n-1}\circ\cdots\circ M_1.$$

The following Proposition is the main result of this section. It compares the compositions of $\f_i's$ and $M_i's.$

\begin{pro}[see also \cite{khanin}]
\label{composicao}  Let $f_i$  be as above. Then for every $C_2 > 0$  there exists $C_3 > 0$ with the following property.  If $\sum_{i=1}^n\delta_i\leq C_2,$ then
$$|\f_1^n-M_1^n|_{C^2}\leq C_3 \cdot (\max_{1\leq j \leq n} \delta_j)^{\nu}.$$
\end{pro}

The proof of Theorem 1 in Khanin and Vul \cite{khanin}  is  the main motivation  to Proposition \ref{composicao}.  Since Proposition \ref{composicao} is not stated explicitly in the paper cited above in its full generality,  we include the full argument for the sake of completeness. Before we prove this  proposition, we need some lemmas.

\begin{lemma}
\label{DZ}
There is $C_4=C_4(C_1,C_2)>0$ such that
$$e^{-C_4} \leq D\f_1^n(x)\leq e^{C_4},$$ for all $x\in [0,1]$ and for all $n\geq 0.$

\end{lemma}

\dem

\begin{eqnarray*}
\ln\frac{D\f_1^n(x)}{D\f_1^n(y)}&=&\ln\frac{D\f_n(\f_1^{n-1}(x))\cdot D\f_{n-1}(\f_1^{n-2}(x))\cdots D\f_1(x)}{D\f_n(\f_1^{n-1}(y))\cdot D\f_{n-1}(\f_1^{n-2}(y))\cdots D\f_1(y)}\\
&=&\sum_{j=1}^n\ln D\f_j(\f_1^{j-1}(x))-\ln D\f_j(\f_1^{j-1}(y))\\
&=&\sum_{j=1}^n\int_{\f_1^{j-1}(y)}^{\f_1^{j-1}(x)}\frac{D^2\f_j(s)}{D\f_j(s)}ds\\
&=&\sum_{j=1}^n\frac{D^2\f_j(z_{j-1})}{D\f_j(z_{j-1})}|\f_1^{j-1}(x)-\f_1^{j-1}(y)|,
\end{eqnarray*}
for some $z_{j-1}\in[\f_1^{j-1}(y), \f_1^{j-1}(x)].$

Therefore by \eqref{onl} we have

$$\left| \ln\frac{D\f_1^n(x)}{D\f_1^n(y)} \right| \leq \sum_{j=1}^n \left|\frac{D^2\f_j(z_{j-1})}{D\f_j(z_{j-1})}\right|\leq C_1\cdot\sum_{j=1}^n \delta_j\leq C_1C_2=C_4.$$
Taking $y\in[0,1]$ such that $D\f_1^n(y)=1$ we have the result.
\cqd

\begin{lemma}
\label{D2Z}
There is $C_5=C_5(C_1,C_2)>0$ such that
$$|D^2\f_1^n(x)|\leq C_5,$$
for all $x\in[0,1]$ and for all $n\geq0.$
\end{lemma}

\dem
Note that
\begin{eqnarray*}
|D^2\f_1^n(x)| & = & \left|\dfrac{D^2\f_1^n(x)}{D\f_1^n(x)}\right|\cdot |D\f_1^n(x)|\\
&= & \left|\sum_{j=1}^n\dfrac{D^2\f_j(\f_1^{j-1}(x))}{D\f_j(\f_1^{j-1}(x))}\cdot D\f_1^{j-1}(x)\right|\cdot   |D\f_1^n(x)|\\
& \leq & e^{2C_4}\sum_{j=1}^n\left|\dfrac{D^2\f_j(\f_1^{j-1}(x))}{D\f_j(\f_1^{j-1}(x))}\right|\\
&\leq & e^{2C_4}\cdot C_1\cdot\sum_{j=1}^n \delta_j \leq  e^{2C_4}\cdot C_1 C_2 = C_5.
\end{eqnarray*}
\cqd

\begin{lemma}
\label{DM}
There are  $C_6=C_6(C_1,C_2),C_7=(C_1,C_2) >0$ such that
$$e^{-C_6} \leq|DM_1^n(x)|\leq e^{C_6}, \  |D^2M_1^n(x)|\leq C_7, \ |D^3M_1^n(x)|\leq C_8.$$
for all $x\in[0,1]$ and for all $n\geq0.$
\end{lemma}

\dem  Since $\mathcal{M}$ is a commutative Lie group, we have $M_1^n = M_N$, where
$N= \sum_{i=1}^n N_{\f_i}$
By Eq. (\ref{onl}) we have $|N_{\f_i}|< C_1 \delta_i$, so
$$|N| \leq C_1\cdot C_2 .$$

One can easily use Eq. (\ref{defMoebius}) to obtain estimates for $D^iM_1^n$, $i=1,2,3$. \cqd

\begin{proof}[Proof of Proposition \ref{composicao}.]

We write

$$\f_1^n-M_1^n=\sum_{i=1}^n M_{i+1}^{n}\circ\f_1^i-M_i^n\circ\f_1^{i-1},$$
where $M_{n+1}^n=f_1^0=\mathrm{Id}.$ Then

\begin{eqnarray*}
|\f_1^n(x)-M_1^n(x)|&\leq& \sum_{i=1}^n \Big|M_{i+1}^n(\f_{i}\circ \f_1^{i-1})(x)-M_{i+1}^n(M_{i}\circ \f_1^{i-1})(x)\Big|\\
&\leq & e^{C_6}\cdot \sum_{i=1}^n \Big|\f_{i}\circ \f_1^{i-1}(x)-M_{i}\circ \f_1^{i-1}(x)\Big|\\
 &\leq & e^{C_6}\cdot C\cdot \sum_{i=1}^n \delta_i^{1+\nu}\\
  &\leq & e^{C_6}\cdot C  \cdot(\max_{1\leq i\leq n}\delta_i)^{\nu}\sum_i \delta_i\\
 &  \leq &   e^{C_6}\cdot C\cdot C_2\cdot(\max_{1\leq i\leq n}\delta_i)^{\nu},
\end{eqnarray*}
where $C>0$ is the constant given by Proposition \ref{moebius}.
\begin{eqnarray*}
\lefteqn{\Big|D\f_1^n(x)-DM_1^n(x)\Big| = }\\
& &=\Big|\sum_{i=1}^n\Big[ DM_{i+1}^n(\f_{i}\circ \f_1^{i-1}(x))\cdot D\f_i(\f_1^{i-1}(x))-{}\\
& &{}-DM_{i+1}^n(M_i\circ \f_1^{i-1}(x))\cdot DM_i(\f_1^{i-1}(x))
\Big] D\f_1^{i-1}(x)\Big|\\
 & &\leq e^{C_4}\cdot \sum_{i=1}^n\Big|DM_{i+1}^n(\f_{i}\circ \f_1^{i-1}(x))\cdot D\f_i(\f_1^{i-1}(x))-{}\\
& &{}-DM_{i+1}^n(M_i\circ \f_1^{i-1}(x))\cdot DM_i(\f_1^{i-1}(x))\Big|.
\end{eqnarray*}
Now, add and subtract the term $DM_{i+1}^n(\f_{i}\circ \f_1^{i-1}(x))\cdot DM_i(\f_1^{i-1}(x))$ in the above expression  to obtain
\begin{eqnarray*}
\lefteqn{\Big|D\f_1^n(x)-DM_1^n(x)\Big| \leq}\\
& &\leq e^{C_4}\cdot e^{C_6}\cdot\sum_{i=1}^n %\Big|DM_{i+1}^n(\Z(f_{i})\circ \Z_{i-1}(f_1)(x))\Big|\cdot
\Big|D\f_i(\f_1^{i-1}(x))-DM_i(\f_1^{i-1}(x))\Big|\\
& &+e^{C_4}\cdot e^{C_6}\cdot\sum_{i=1}^n \Big| DM_{i+1}^n( \f_{i}\circ \f_1^{i-1}(x))-DM_{i+1}^n(M_i\circ \f_1^{i-1}(x))\Big|\\
& & \leq  e^{C_4}\cdot e^{C_6}\cdot C\cdot\sum_{i=1}^n \delta_i^{1+\nu} + e^{C_4}\cdot e^{C_6} \cdot C_7 \cdot C  \cdot\sum_{i=1}^n \delta_i^{1+\nu}\\
 & &\leq e^{C_4}\cdot e^{C_6}\cdot (1+C_7)\cdot C\cdot  C_2\cdot(\max_{1\leq i\leq n}\delta_i)^{\nu}.
\end{eqnarray*}

Now note that
$$D^2\f_1^n(x)-D^2M_1^n(x)=(\mathrm{I})+(\mathrm{II})+(\mathrm{III}),$$
where
\begin{eqnarray*}
(\mathrm{I})&=&\sum_{i=1}^nD^2M_{i+1}^n(\f_{i}\circ \f_1^{i-1}(x))\cdot \Big(D\f_i(\f_1^{i-1}(x))\Big)^2\cdot \Big(D\f_1^{i-1}(x)\Big)^2\\
& &-\sum_{i=1}^nD^2M_{i+1}^n(M_i\circ \f_1^{i-1}(x))\cdot \Big(DM_i(\f_1^{i-1}(x))\Big)^2\cdot \Big(D\f_1^{i-1}(x)\Big)^2,
\end{eqnarray*}
\begin{eqnarray*}
(\mathrm{II})&=&\sum_{i=1}^nDM_{i+1}^n(\f_{i}\circ \f_1^{i-1}(x))\cdot D^2\f_i(\f_1^{i-1}(x))\cdot \Big(D\f_1^{i-1}(x)\Big)^2\\
& &-\sum_{i=1}^nDM_{i+1}^n(M_{i}\circ \f_1^{i-1}(x))\cdot D^2M_i(\f_1^{i-1}(x))\cdot \Big(D\f_1^{i-1}(x)\Big)^2
\end{eqnarray*}
and
\begin{eqnarray*}
(\mathrm{III})&=&\sum_{i=1}^nDM_{i+1}^n(\f_{i}\circ \f_1^{i-1}(x))\cdot D\f_i(\f_1^{i-1}(x))\cdot D^2\f_1^{i-1}(x)\\
& &-\sum_{i=1}^nDM_{i+1}^n(M_{i}\circ \f_1^{i-1}(x))\cdot DM_i(\f_1^{i-1}(x))\cdot D^2\f_1^{i-1}(x).
\end{eqnarray*}

In $(\mathrm{I})$ we first add  subtract the term
$$\sum_{i=1}^nD^2M_{i+1}^n(\f_{i}\circ \f_1^{i-1}(x))\cdot \Big(DM_i(\f_1^{i-1}(x))\Big)^2\cdot \Big(D\f_1^{i-1}(x)\Big)^2$$ and then we use Lemma \ref{DZ} and Lemma \ref{DM} with estimates for the first derivative of $\f_1^n$ and $M_1^n,$ to obtain

\begin{eqnarray}
\label{1}
|(\mathrm{I})| \leq  2 \cdot \max\{C_9,C_{10}\}\cdot(\max_{1\leq i\leq n}\delta_i)^{\nu},
\end{eqnarray}
where  $C_9=C_2\cdot C_7\cdot C\cdot e^{2C_4}(e^{C_4}+e^{C_6})$ and $C_{10}=C\cdot C_2\cdot C_8 \cdot e^{2C_4}\cdot e^{C_6}.$

In $(\mathrm{II})$, we first  add and subtract the term
$$\sum_{i=1}^nDM_{i+1}^n(M_{i}\circ \f_1^{i-1}(x))\cdot D^2\f_i(\f_1^{i-1}(x))\cdot \Big(D\f_1^{i-1}(x)\Big)^2$$ and then we use Lemma \eqref{DZ} and Lemma \eqref{DM} with estimates for the first derivative of $\f_1^n$ and $M_1^n,$ to obtain

\begin{eqnarray}
\label{2}
|(\mathrm{II})| \leq 2 \cdot \max\{C_{11},C_{12}\} \cdot(\max_{1\leq i\leq n}\delta_i)^{\nu},
\end{eqnarray}
where $C_{11}=C\cdot C_2\cdot C_5\cdot C_7\cdot e^{2C_4}$ and $C_{12}=C\cdot C_2\cdot e^{2C_4+C_6}.$

Finally we add and subtract  the expression
$$\sum_{i=1}^nDM_{i+1}^n(\f_{i}\circ \f_1^{i-1}(x))\cdot DM_i(\f_1^{i-1}(x))\cdot D^2\f_1^{i-1}(x)$$ in $(\mathrm{III})$ and use again   Lemma \ref{DZ} and Lemma \ref{DZ}, obtaining

\begin{eqnarray}
\label{3}
|(\mathrm{III})| \leq2 \cdot \max\{C_{13},C_{14}\}\cdot(\max_{1\leq i\leq n}\delta_i)^{\nu},
\end{eqnarray}
where $C_{13}=C\cdot C_2\cdot C_5\cdot e^{C_6}$ and $C_{14}=C\cdot C_2\cdot C_5\cdot C_7\cdot e^{C_6}.$

Taking $C_3=6 \cdot \max\{C_9,C_{10},C_{11},C_{12},C_{13},C_{14}\}$ we get the result.
\end{proof}

\section{Renormalization for genus one g.i.e.m. }\label{rensec}

Let  $f \in \mathcal{B}^{2+\nu}_k$ be a g.i.e.m. with $d$ intervals.   If $f$ has only one discontinuity then if we identify the endpoints of its domain $I$ we obtain a piecewise smooth homeomorphism of the circle with irrational rotation. So we can apply the Denjoy Theorem for these maps, proving the non existence of wandering intervals. That is a main difference between genus on g.i.e.m. and those with higher genus (even piecewise affine g.i.e.m. with higher genus can have wandering intervals).  In this section we will study this relation between genus one g.i.e.m. and piecewise homeomorphisms of the circle in more details. The combinatorial analysis next somehow extends the results of Nogueira and Rudolph \cite{nogueira}.

For the sake of simplify the notation, assume  that $f$ has only one discontinuity. Note that $f$ can be seen as a g.i.e.m. in $f \in \mathcal{B}^{0+1}_k$ with {\it two} intervals.  Indeed, if $I_\alpha=(c_\alpha,d_\alpha)$, where $d_\alpha$ the the unique point of discontinuity of $f$, define
$$J_A:= \overline{\cup_{\pi_0(\beta)\leq \pi_0(\alpha)}  I_\beta},$$
$$J_B:= \overline{\cup_{\pi_0(\beta)> \pi_0(\alpha)}  I_\beta}.$$

Then $(f,\{A,B\},\{J_A,J_B\})$ is a g.i.e.m. with two intervals.   We can either renormalize as a g.i.e.m. with $d$ intervals , denoted by
$$R_d(f), R_d^2(f),R_d^3(f), \dots$$
 or as a g.i.e.m. with two intervals, denoted
 $$R_2(f), R_2^2(f),R_2^3(f), \dots$$
 We call $R_d^i$ the $i$-th $d$-renormalization of $f$ and $R_2^i$ the $i-$th $2$-renormalization of $f$.
If we see $f$ as a homeomorphism of the circle then we can do the usual renormalization of the circle. These sequence of renormalizations, denoted $R_{rot}^i(f)$ turns out to be just a acceleration of the Rauzy-Veech induction consisting   in the subsequence of $2$-renormalizations $R_2^{n_i}(f)$ defined in the following way:  $n_0=0$ and $n_{i+1}$ is the first $n> n_{i}$ whose type is distinct from the type of the $n_i$-th $2$-renormalization.

The relation between the $d$ and $2$-renormalizations is given by the following proposition

\begin{pro} \label{subseq}Let $f$ be a genus one g.i.e.m with $d$ intervals in $\mathcal{B}_k$ with only one discontinuity,  where $\pi_1(\alpha_0)=1$ and $\pi_0(\alpha_1)=1$. One of the two cases occurs
\begin{itemize}

\item[(A)] We have
$$\overline{\cup_{\pi_1(\beta)\geq \pi_1(\alpha_1)}f( I_\beta) } \subset \overline{\cup_{\pi_0(\beta)\geq \pi_0(\alpha_0)} I_\beta}.$$Then $f$ is $2$-renormalizable of type $0$ and $R_2(f)=R_d^n(f)$, where $n$ is the first $d$-renormalization where the letter $\alpha_*$ wins from  letter $\alpha_0.$ Here $\al_*$ is such that $f(c_{\al_0})\in I_{\al_*}.$
\item[(B)] We have
$$ \overline{\cup_{\pi_0(\beta)\geq \pi_0(\alpha_0)} I_\beta}\subset \overline{\cup_{\pi_1(\beta)\geq \pi_1(\alpha_1)}f( I_\beta) }.$$Then $f$ is $2$-renormalizable of type $1$ and $R_2(f)=R_d^n(f)$, where $n$ is the first $d$-renormalization where the letter $\alpha_1$ wins  from  letter $\alpha_*$. Here $\al_*$ is such that $c_{\al_*}\in f(I_{\al_1}).$
\end{itemize}
\end{pro}

\dem We are going to prove the Claim A. The proof of the Claim B is similar.
It is easy to see that when  the letter $\al_*$ wins from the letter $\al_0$ for the first time, it wins  with type 0. Using the notation defined in Eq.  \eqref{mono} the Rauzy-Veech algorithm is given by

\begin{eqnarray}
\label{pt0}
& & (p(1),p(2),\ldots, p(s), \ldots, p(d)) \stackrel{0}{\rightarrow} (p(2),\ldots, p(1), p(s),\ldots, p(d) ),
\end{eqnarray}
where $s$ is such that $p(s)=1,$ and

\begin{eqnarray}
\label{pt1}
(p(1),\ldots, p(r), \ldots, p(d)) \stackrel{1}{\rightarrow} (p(1),\ldots, p(r), p(d),\ldots ),
\end{eqnarray}
where $r$ is such that $p(r)=d.$

As by assumption  $f$ has only one discontinuity we have that $p=(k...d \ 1 ... k-1),$ where $\pi_1(\al_1)=k.$
%\begin{center}
%\xymatrix{
%& (k+1...d \ k \ 1...k-1)\\
%(k...d \ 1...k-1)\ar[dr]^1\ar[ur]^0 & \\
%& (k...d \ k-1 \ 1...k-2),\\
%}
%\end{center}

 We assert that iterating the algorithm $N$ times, with  $N=s+r\leq n-1,$ where $n$ is such that the letter $\al_*$ wins from the letter $\al_0$ for the first time, we obtain that
 \begin{eqnarray}
 \label{pN}
 & & p^N=(k+s,\ldots,d,\ k-r, \ldots, k+s-1, \ 1,\ldots,k-r-1),
 \end{eqnarray}
 where $0\leq s\leq d-k$ and $0\leq r \leq k-1$ are such that
 $$s=\#\{\ve_m=0: 0\leq m\leq N\} \ \text{ and } \ r=\#\{\ve_m=1: 0\leq m\leq N\}.$$

 For $N=1$ the assertion is true because $s=0$ and $r=1$ or $s=1$ and $r=0.$ Assume that the formula $\eqref{pN}$ holds for $N-1.$ Then by formulas $\eqref{pt0}$ and $\eqref{pt1}$ the assertion holds for $N.$

We  know that $\ve_n=0$ and that $p^{n-1}(1)=d.$ So

 $$p=p^0=(k...d \ 1 ... k-1)\stackrel{\ve_0}{\rightarrow} \cdots \stackrel{\ve_{n-1}}{\rightarrow} p^{n-1}=(d \ j...d-1 \ 1...j-1).$$

Then $p^n=(j...d-1\ d \ 1...j-1)$ which completes the proof.

\cqd

\begin{cor} Let $f$ be a genus one g.i.e.m with $d$ intervals in $\mathcal{B}_k$ with only one discontinuity. Then there exists a sequence $m_i < m_{i+1}$ such that $m_{i+1}-m_i< d$ and
$$R^i_2(f)=R^{m_i}_d(f).$$

\end{cor}

%There is $0\leq n \leq k$ such that $R^n(f)$ is a circle homeomorphism. Since $f$ has no periodic points then neither $R^n(f)$  does. Then we have the following proposition

%\begin{pro}
%Let $f$ be a generalized i.e.m. such that $\gamma(f)$ is $k-$bounded. Then there is $n\in\mathbb{N}$ (and therefore infinite such $n$) such that $R^n(f)$ is conjugated to a rigid rotation with irrational rotation number.
%\end{pro}

%The $n's$ given by the previous proposition will be called from {\it time of conjugation}.

%Suppose that $f$ is a circle homeomorphism and consider two types of renormalization: $R_{\mathrm{rot}}(f),$ $R_2(f)$ which are the usual renormalization of the circle, this is, the Zorich renormalization and the Rauzy-Veech renormalization, respectively, looking $f$ as generalized i.e.m. with two intervals.

The next result gives us a relationship between $R_{\mathrm{rot}}(f),$ $R_2(f)$ and $R(f).$

\begin{pro}
Let $f$ be a g.i.e.m. such that $\gamma(f)$ is $k-$bounded. Then for all $i\geq0$
$$R_{\mathrm{rot}}^i(f)=R^{k_i}_2(f)=R^{n_{{k}_{i}}}_d(f).$$
\end{pro}
\dem
The first equality follows by definition of $R_{\mathrm{rot}}$ and $R_2.$ The second equality follows da Proposition \ref{subseq}.
\cqd
\subsection{Bounded Geometry for maps in $\mathcal{B}_k^{2+\nu}$} A classical result on the circle homeomorphisms of class $P$ (absolutely continuous homeomorphisms on the circle with bounded variation derivative) is the following lemma, whose demonstration can be found in Herman \cite{herman}.

\begin{lemma}
\label{denjoy}
Let $f$ be a g.i.e.m. with genus. Let $n_0$ be the first $n$ such that $R^{n_0}_df$ has only one discontinuity and define $n_i$ such that $R_d^{n_i}f = R_{rot}^i(R_d^{n_0}f)$. Then for all $i\geq0$ and $\al\in\A,$
$$\exp(-V)\leq DR^{n_{i}}_df(x)\leq\exp(V) \mbox{~for all~} x\in I^{n_i}_{\al},$$
where $V=\mathrm{Var}(\log Df).$
\end{lemma}

%\begin{obs}
%It is easy to see that the only permutations invariant by renormalization are permutations where $f$ is a circle homeomorphism. As $\gamma(f)$ is $k-$bounded we have that $f$ may have a limited number of loops these permutations. Suppose now that $f$ with the combinatorial data $\pi$ is circle homeomorphism and $R(f)$ with $\pi^1$ is not. Then the maximum time that $R^n(f)$ will be back to a circle homeomorphism is $d-1.$ Therefore $n_{{k}_{i+1}}-n_{{k}_{i}}$ is uniformly limited, this is  $|n_{{k}_{i+1}}-n_{{k}_{i}}|\leq C_{15}$ for all $i\geq0,$ where $C>0$ is constant.
%\end{obs}

\begin{lemma}
\label{strong}
Let $f\in\mathcal{B}_k^{2+\mu}$. There is $C_{16}>0$ such that
$$\exp(-C_{16}V)\leq DR_d^nf(x)\leq\exp(C_{16}V) \mbox{~for all~} x\in I^{n}, \ n \in \mathbb{N}.$$
\end{lemma}
\dem Because $f$ has $k$-bounded combinatorics, there exists $C$ with the following property: Let $i\geq0$ be  such that $n_{{k}_{i}}\leq n<n_{{k}_{i+1}}.$ Then for every $\alpha \in \mathcal{A}$ there exists $a \leq C$ such that $(R^n_d f)(x)=(R^{n_{{k}_{i}}}_d f)^a(x)$, for $x \in I^n_\alpha$. Now the lemma  follows from Lemma \ref{denjoy}. \cqd

\begin{lemma}[Non-Colapsing Domains]
\label{domain}
Let $f\in\mathcal{B}_k^{2+\mu}.$ There is  $C_{17}>1$ such that
$$\dfrac{1}{C_{17}}\leq \dfrac{|I_{\al}^n|}{|I_{\beta}^n|}\leq C_{17}, \mbox{~~for all~~} \al,\beta\in\A \mbox{~and~} n\geq0.$$

\end{lemma}

\dem Note that by Lemma \ref{strong}  we have that for all $\alpha \in \mathcal{A}$
\begin{equation} \label{compare} \exp(-C_{16}V)\leq\frac{|R^n(f)(I^n_\alpha)|}{|I^n_\alpha|}\leq\exp(C_{16}V).\end{equation}
We claim that if the letter $\alpha_\star$ is the winner in the $n-$level  then
$$|I^n_{\alpha_\star}|, |R^n(f)(I^n_{\alpha_\star})| \leq \frac{k+1-\exp(-(2^k+1)C_{16}V)}{k+1} |I^n|.  $$
Indeed, otherwise by (\ref{compare})
$$|I^n_{\alpha}|, |R^n(f)(I^n_{\alpha})|<  \frac{\exp(-(2^k+1)C_{16}V)|I^n|}{k+1}< \frac{\min \{ |I^n_{\alpha_\star}|, |R^n(f)(I^n_{\alpha_\star})|\}}{\exp(2^kC_{16}V)k}.$$
for every $\alpha\neq \alpha_\star$. As a consequence the letter $\alpha_\star$ will be the winner for at least $k$ consecutive times, which contradicts $f \in \mathcal{B}_k$. So there exists $\delta \in (0,1)$ such that

\begin{equation}\label{dec_int}\frac{|I^{n+1}|}{|I^n|}\geq 1-\delta\end{equation}
for every $n$. Note that by (\ref{compare}) we have
$$|I^{n+1}_\alpha|\leq \exp(C_{16}V) |I^n_\alpha|,$$
%$$|R^{n+1}(f)I^{n+1}_\alpha|\leq exp(C_{16}V) |R^{n}(f)I^n_\alpha|$$
for every $n$ and $\alpha$. Moreover if $\alpha_\star$ is the winner and $\beta_\star$ is the loser at the $n$th level we have
$$|I^{n+1}_{\beta_\star}|\leq \exp(2C_{16}V) |I^{n}_{\alpha_\star}|.$$

So fix $\alpha, \beta \in \mathcal{A}$. Since $\beta$ loses transitively from $\alpha$ after at most $k$ renormalizations, we obtain
\begin{equation} \label{est_trans} |I^{n+k}_\beta|\leq \exp(2kC_{16}V) |I^{n}_\alpha|,\end{equation}
for every $n$, $\alpha$ and $\beta$. We claim that
$$ |I^{n}_\alpha| \geq \frac{(1-\delta)^k\exp(-2kC_{16}V) }{d} |I^n|.$$
Indeed, otherwise by (\ref{est_trans})
$$|I^{n+k}|=\sum_\beta |I^{n}_\beta|\leq (1-\delta)^k |I^n|,$$
which contradicts (\ref{dec_int}).
 \cqd

\begin{lemma}[Non-Colapsing Images]
\label{image}
Let $f\in\mathcal{B}_k^{2+\mu}.$ There is a constant $C_{18}>1$ such that
$$\dfrac{1}{C_{18}}\leq \dfrac{|R^n(f)(I_{\al}^n)|}{|R^n(f)(I_{\beta}^n)|}\leq C_{18}, \mbox{~~for all~~} \al,\beta\in\A \mbox{~and~} n\geq0.$$
\end{lemma}

\dem Follows directly from Lemmas \eqref{strong} and \eqref{domain}.
\cqd

\begin{pro}
\label{corte}
Let $f\in\mathcal{B}_k^{2+\mu}$ and let $\al,\beta\in\A$ be the winner and loser letters of $R^n(f),$ respectively. Then there is $0<\la_1<\la_2<1,$ such that
%$$\la_1<\dfrac{|f^{(1-\varepsilon)q_n^{\beta}}(I_{\beta}^n)|}{|f^{\varepsilon q_n^{\al}}(I_{\al}^n) |}<\la_2,$$
\begin{equation} \label{quo} \la_1<\dfrac{|(R^nf)^{1-\varepsilon}(I_{\beta}^n)|}{|(R^nf)^{\varepsilon}(I_{\al}^n) |}<\la_2,\end{equation}
where $\varepsilon\in\{0,1\}$ is the type of $R^n(f).$
\end{pro}

\dem  If the quotient in (\ref{quo}) is too close to $0$ then $(R^nf)^{1-\varepsilon}(I_{\beta}^n)$ is very small compared to $I^n$, which contradicts either Lemma \ref{domain} or Lemma  \ref{image}.   If the quotient in (\ref{quo}) is too close to $1$ then  $|(R^{n+1}f)^\varepsilon(I_{\al}^{n+1})|=|(R^nf)^{\varepsilon}(I_{\al}^n)| -|(R^nf)^{1-\varepsilon}(I_{\beta}^n)|$ is very small compared to $I^{n+1}$, what again contradicts either Lemma \ref{domain} or Lemma \ref{image}.
\cqd

By definition of renormalization operator we know that
$$[0,1)=\bigvee_{\al\in\A}~\bigvee_{i=0}^{q_n^{\al}-1}f^i(I^n_{\al}),$$
where $\bigvee$ means disjoint union. Thus the elements $f^i(I^n_{\al})$ for all $\al\in\A$ and for all $0\leq i\leq q_n^{\al}-1$ form a partition which we denote by $\mathcal{P}^n.$ The norm of  $\mathcal{P}^n$ is defined by

$$|\mathcal{P}^n|=\max_{ {\al\in\A} \atop {0\leq i\leq q_n^{\al}-1}} ^{} \{|f^i(I^n_{\al})|\}.$$

The next result says that $|\mathcal{P}^n|$ tends to zero exponentially fast.

\begin{pro}
\label{pvz}
Let $f\in\mathcal{B}_k^{2+\nu}.$ Then for n sufficiently large there is  $\la=\la(\la_1,\la_2)$ with $0<\la<1$ such that
$$|\mathcal{P}^{n+k}|\leq\lambda\cdot |\mathcal{P}^n|.$$
\end{pro}

\dem Let $f^{i_{n+k}}(I^{n+k}_\alpha)\in\mathcal{P}^{n+k}.$ There are $\al_j\in\A$ and $0\leq i_j\leq q_{j}^{\alpha_j}-1$ for $n\leq j\leq n+k$ such that
$$f^{i_n}(I^n_{\al_n})\supset \cdots\supset f^{i_j}(I^{j}_{\al_{j}}) \supset f^{i_{j+1}}(I^{j+1}_{\al_{j+1}})\cdots    \supset f^{i_{n+k}}(I^{n+k}_{\al_{n+k}}).$$

We claim that there exists $j_0$ such that   $\alpha_{j_0}$ is the winner in the $j_0$th level. Indeed if $\alpha=\alpha_n=\cdots=\alpha_{n+k}$, let $j_0$ be a level between levels $n$ and $n+k$ such that $\alpha$ wins. Such level exists because $f \in \mathcal{B}_k$.  Otherwise there exists $j_0$ such that $\alpha_{j_0+1}\neq \alpha_{j_0}$.  This is only possible if $\alpha_{j_0}$ is the winner and $\alpha_{j_0+1}$ the loser  in the $j_0$-th level.  By Proposition \ref{corte}  and Lemma \ref{strong} we have

$$\frac{|f^{i_{n+k}}(I^{n+k}_{\al_{n+k}})|}{|f^{i_n}(I^n_{\al_n})|} \leq \frac{|f^{i_{j_0+1}}(I^{j_0+1}_{\al_{j_0+1}})|}{|f^{i_{j_0}}(I^{j_0}_{\al_{j_0}})| }\leq \lambda < 1.$$
for some $\lambda$ that depends only on $V=Var(\log Df)$, $\lambda_1$ and $\lambda_2$.
\cqd

\begin{proof}[Proof of Theorem \ref{teo1}.]  Note that
$$Z_{I^n_\alpha}(R^n(f))=       Z_{f^{q_{\alpha}^n-1}(I^n_\alpha)}(f) \circ        \cdots     \circ   Z_{f(I^n_\alpha)}(f) \circ  Z_{I^n_\alpha}(f). $$
The intervals $f^{i}(I^n_\alpha)$, $i=0, \dots, q^n_\alpha-1$ belong to the partition $\mathcal{P}^n$. In particular
$$\sum_{i=0}^{q^n_\alpha-1}  |f^{i}(I^n_\alpha)| \leq 1 $$
and by Proposition \ref{pvz}
$$\sup_{0\leq i < q^n_\alpha}  |f^{i}(I^n_\alpha)| \leq |\mathcal{P}^n|\leq \lambda^{n/k-1}.$$
So we can apply  Proposition \ref{composicao} to obtain that
$$|Z_{I^n_\alpha}(R^n(f)) - M_1^n|_{C^2}\leq C_3 \cdot \lambda^{\mu(n/k-1)}.$$
Recall that $M_a\circ M_b=M_{a+b}.$  So
$$M_1^n=M_N,$$
where (see (\ref{zoomnon}) and (\ref{noncomp}) )
$$N = \sum_{i=0}^{q^n_\alpha-1} N_{ Z_{f^i(I^n_\alpha)}(f)}= \sum_{i=0}^{q^n_\alpha-1} \int_{f^i(I^n_\alpha)} \frac{D^2 f(x)}{Df(x)} \ dx= \int_{I^n_\al}\frac{D^2 R^n(f)(x)}{D R^n(f)(x)}dx  \quad \forall \al\in\A.$$
\end{proof}
\section{Symbolic representation}\label{synsec}
To prove Theorem 2 and 3 we need a finer understanding of the geometry of the partitions generated by the sequence of renormalizations. To this end we will introduce a certain symbolic representation for the dynamics, that is somehow a generalization of the symbolic representation introduced by Sinai and  Khanin \cite{khanin1}. Consider the set of {\it letters}
$$\mathcal{L}= \{   (\alpha,\ve,n) \ s.t. \ \alpha \in \mathcal{A}, \ \ve \in \{  0,1\}, \ n \in \mathbb{N} \}.$$
Define $\pi_3(\alpha,\ve,n)=n$, $\pi_2(\alpha,\ve,n)=\ve$. We will use the notation $a_i$ for $a_i \in \mathcal{L}$ such that $\pi_n(a_i)=i$.

In this section we construct a symbolic representation for the dynamics  of a g.i.e.m. $f\in \mathcal{B}_k$. For each $n$ we consider the partition of $[0,1]$ given by
$$ \tilde{\mathcal{P}}^n=\{   f^i(I^n_{\al}) \ s.t. \  \al\in\A \ and \  1\leq i\leq q_n^{\al} \}.$$
Let $$\Lambda=  [0,1]\setminus \cup_n\cup_{J \in \tilde{\mathcal{P}}^n} \partial J.$$
We will define a function
$$s\colon \Lambda \rightarrow \mathcal{L}^{\mathbb{N}},$$
in the following way. We have $s(x)=(a_i)_{i\in  \mathbb{N}}$, $a_i \in \mathcal{L}$. If $i=0$ then $x \in f(I^0_\beta)$ for some $\beta$  and we define $a_0=(\beta,0,0)$.  If $i > 0$, let
$$k_{i-1}(x):= \min \{ k \geq 0\ s.t. \ f^k(x) \in I^{i-1}  \}.$$
Then either  $f^{k_{i-1}}(x) \in I^i$, so $f^{k_{i-1}}(x) \in f_i(I_\beta^i)$ for some $\beta$ and we define  $a_i=(\beta,0,i)$, or $f^{k_{i-1}}(x) \in I^{i-1}\setminus I^i$, so $f_{i-1}(f^{k_{i-1}}(x)) \in  f_i(I_\beta^i)$ for some $\beta$ and we define  $a_i=(\beta,1,i)$. Note that in any case $f^{k_i(x)}(x) \in I^i_\beta$ and $\pi_2(a_i)=0$ if and only if $k_i(x)=k_{i-1}(x)$.

\subsection{Admissible cylinders and its properties} Given a finite subset $S=\{n_1,n_2, \dots, n_k\}\subset \mathbb{N}$, with $\#S=k$ and  a finite sequence $a_{n_1},\dots,a_{n_k} \in \mathcal{L}$ , with  $\pi_3(a_{n_i})=n_i$ we can consider the {\it word} $$\omega= a_{n_k}a_{n_{k-1}}\dots a_{n_1}.$$

For each word we can define the cylinder
$$[\omega] = [a_{n_k}a_{n_{k-1}}\dots a_{n_1}]=\overline{\{x \in \Lambda \ s.t. \ s_{n_i}(x)=a_{n_i}, \ 1\leq i\leq k   \}}.$$
If this cylinder is not empty we will say that the word $\omega$ is {\it admissible}.  Indeed we can give a definition of admissible words just in terms of the combinatorial data of the g.i.e.m. $f$.

We claim  that the set whose elements are the closures of  intervals in
$\tilde{\mathcal{P}}^n$
is exactly   the set of  admissible cylinders of the form $[a_na_{n-1}\dots a_0]$.  Indeed when $n = 0$ we have $\tilde{\mathcal{P}}^0=\{ f(I_{\al})\}_{\al \in \mathcal{A}}.$ Then
$$\overline{f(I_{\al})}=[(\al,0,0)]$$
Suppose  by induction that we have verified that the set of   all elements $\overline{f^i(I^n_{\al})}$, $1\leq i\leq q^n_\alpha$,  is the set of admissible  cylinders of the form  $[a_na_{n-1}\dots a_0]$. Recall that $\al^n(\ve)$,  $\al^n(1-\ve) \in \mathcal{A}$ are  the winner and loser, respectively, if $f_n$ has  type $\ve.$ There are three cases:

\begin{itemize}
\item If  $\al\not=\al^n(1-\ve)$ then $I^{n+1}_\alpha \subset I^n_\alpha$ and $q^{\al}_{n+1}=q^{\al}_n$. So  $$\overline{f^i(I^{n+1}_{\al})}=[a_{n+1} a_n \dots a_0] $$
for every  $1\leq i\leq q^{\al}_{n+1}=q^{\al}_n$, where $a_{n+1}=(\al,0,n+1)$ and $\overline{f^i(I^{n}_{\al})}=[a_n\dots a_0]$;
\item If $\al=\al^n(1-\ve)$ and  $1\leq i\leq q_n^{\al^n(\ve)}$ we have
$$\overline{f^{i+(1-\ve)q_n^{\al^n(1-\ve)}}(I^{n+1}_{\al^n(1-\ve)})}=[a_{n+1} a_n \dots a_0] $$
where $a_{n+1}=(\al^n(1-\ve),\ve,n+1)$ and   $\overline{f^i(I^{n}_{\al^n(\ve)})}=[a_n\dots a_0];$

\item If $\al=\al^n(1-\ve)$ and   $1\leq j\leq q_n^{\al^n(1-\ve)}$ we have
$$\overline{f^{j+\ve q_n^{\al^n(\ve)}}(I^{n+1}_{\al^n(1-\ve)})}=[a_{n+1}a_n\dots a_0],$$ where $a_{n+1}=(\al^n(1-\ve),1-\ve,n+1)$ and $\overline{f^j(I^{n}_{\al^n(1-\ve)})}=[a_n\dots a_0].$

\end{itemize}

As a consequence, for any admissible word of the form $a_n\dots a_0$, with $a_n=(\alpha,\chi, n)$ we have that
the first entry times $k_1$, $k_2$, $\dots$,  $k_n$ are constant functions  in $[a_n\dots a_0]$ and $f^{k_n}[a_n\dots a_0]=f_n(I^n_\alpha)$.

Denote by $\ell(a_n \ldots a_0)$  the Lebesgue measure of the cylinder $[ a_n \ldots a_0]$. The  proof of Theorem 2 will be based on the ergodic properties of the sequence of random variables $a_n=(\al_n, \chi_n,n),$ $\al_n\in\mathcal{A}$ and $\chi_n\in\{0,1\}$ with respect to the Lebesgue measure.

If the word $a_{n-1}\ldots a_0$ is admissible we can consider the conditional probabilities
$$\ell(a_n|a_{n-1}\ldots a_0)=\dfrac{\ell(a_n\ldots a_0)}{\ell(a_{n-1}\ldots a_0)}.$$

\begin{lemma}
\label{difeocil}
Let  $$\omega_1=a'_0 \ldots a'_{n-1} a_{n}\ldots a_{n+k},$$ $$ \tilde{\omega}_1= a''_0 \ldots a''_{n-1} a_{n}\ldots a_{n+k}$$  be admissible words. Denote    $\omega_2=a'_0 \ldots a'_{n-1}$ and $ \tilde{\omega}_2=a''_0 \ldots a''_{n-1}$.  Then
\begin{itemize}
\item[A.]  Indeed $\pi_1(a'_{n-1})=\pi_1(a''_{n-1})=:\beta$ and there exist $1\leq i,j\leq q^{n-1}_\beta$  such that  $[\omega_2]= \overline{f^i(I^{n-1}_\beta)}$ and $[\tilde{\omega}_2]=\overline{ f^j(I^{n-1}_\beta)}$,
\item[B.] In particular $r=j-i$  is the  unique $r\in \mathbb{Z}$  such that $f^r$ is a diffeomorphism on $[\omega_2]$ and $f^r([\omega_2])=[\tilde{\omega}_2]$.
\item[C.] The integer  $r$ is also  the unique integer such that  $f^r$ is a diffeomorphism on $[\omega_1]$ and $f^r([\omega_1])=[\tilde{\omega}_1]$.
\end{itemize}
\end{lemma}
\dem The uniqueness of $r$ follows from  the fact that $f$ does not have periodic points. Indeed if $f^{r_1}$ and $f^{r_2}$, $r_1 < r_2$, are diffeomorphisms on $[\omega_i]$ and  $f^r([\omega_i])=[\tilde{\omega}_i]$ for some $i \in \{1,2\}$ then the points in $\partial [\omega_i]$ are fixed points of $f^{r_2-r_1}$, which is a contradiction. It remains to show  the existence of $r$. We will prove this by induction on $k$. Suppose $k=0$. Denote $a_n=(\al,\chi,n)$.  Let $\ve $ be the type of $f_{n-1}$. By the previous discussion about the partitions $\tilde{\mathcal{P}}^n$, there are three cases. \\

\noindent {\it Case i.}  If $\al\not=\al^{n-1}(1-\ve)$ then $\chi=0$ and
$$[a_{n} a'_{n-1} \dots a'_0]= \overline{f^i(I^{n}_{\al})}$$
$$[a_{n} a''_{n-1} \dots a''_0]= \overline{f^j(I^{n}_{\al})}$$ for some  $1\leq i, j\leq q^{\al}_{n}=q^{\al}_{n-1}$, with $\overline{f^i(I^{n-1}_{\al})}=[a'_{n-1}\dots a'_0]$ and  $\overline{f^j(I^{n-1}_{\al})}=[a''_{n-1}\dots a''_0]$. In particular $\alpha=\pi_1(a'_{n-1})=\pi_1(a''_{n-1})$.  So take  $r=i-j$. \\

\noindent {\it Case ii.}  If  $\al =\al^{n-1}(1-\ve)$ and $\chi=\ve$ then
$$\overline{f^{i+(1-\ve) q_{n-1}^{\al}}(I^{n}_{\al})}=[a_{n}a'_{n-1}\dots a'_0],$$
$$\overline{f^{j+(1-\ve) q_{n-1}^{\al}}(I^{n}_{\al})}=[a_{n}a''_{n-1}\dots a''_0],$$ for some $1\leq i, j\leq q_{n-1}^{\al^{n-1}(\ve)}$ with
 $[a'_{n-1}\dots a'_0]=\overline{f^i(I^{n-1}_{\al^{n-1}(\ve)})}$ and $[a''_{n-1}\dots a''_0]=\overline{f^j(I^{n-1}_{\al^{n-1}(\ve)})}$.  In particular $\alpha^{n-1}(\ve)=\pi_1(a'_{n-1})=\pi_1(a''_{n-1})$.  So take  $r=i-j$. \\

 \noindent {\it Case iii.}  If  $\al =\al^{n-1}(1-\ve)$ and $\chi=1-\ve$ then
$$\overline{f^{i+\ve q_{n-1}^{\al^{n-1}(\ve)}}(I^{n}_{\al})}=[a_{n}a'_{n-1}\dots a'_0],$$
$$\overline{f^{j+\ve q_{n-1}^{\al^{n-1}(\ve)}}(I^{n}_{\al})}=[a_{n}a''_{n-1}\dots a''_0],$$ for some $1\leq i, j\leq q_{n-1}^{\al}$ with
 $[a'_{n-1}\dots a'_0]=\overline{f^i(I^{n-1}_{\al})}$ and $[a''_{n-1}\dots a''_0]=\overline{f^j(I^{n-1}_{\al})}$. Again we have $\alpha=\pi_1(a'_{n-1})=\pi_1(a''_{n-1})$.Take  $r=i-j$.  \\

\noindent This completes the proof for $k=0$. Suppose by induction we have proved the statement for $k-1\geq 0$.  By the case $k=0$  there exists a unique $r$ such that $f^r[a'_0 \ldots a'_{n-1} a_{n}\ldots a_{n+k-1}]= [a''_0 \ldots a''_{n-1} a_{n}\ldots a_{n+k-1}]$ in a diffeomorphic way, and moreover $r$ is the unique $r$ such that $f^r[a'_0 \ldots a'_{n-1} a_{n}\ldots a_{n+k}]= [a''_0 \ldots a''_{n-1} a_{n}\ldots a_{n+k}]$. By induction assumption there exists a unique $r'$ such that  $f^{r
}[a'_0 \ldots a'_{n-1}]= [a''_0 \ldots a''_{n-1}]$ and moreover $r'$ is the unique integer such that $$f^r[a'_0 \ldots a'_{n-1} a_{n}\ldots a_{n+k-1}]= [a''_0 \ldots a''_{n-1} a_{n}\ldots a_{n+k-1}].$$ So $r=r'$. This completes the proof.  \cqd

\begin{lemma}
\label{markov}
There are constants $C_{19}>0$ and $0<\la_3=\la_3(\la)<1$ such that
$$e^{-C_{19}\la_3^s}\leq \frac{\ell(a_n|a_{n-1}\ldots a_{n-s} a''_{n-s-1} \ldots a''_0)}{\ell(a_n|a_{n-1} \ldots a_{n-s} a'_{n-s-1} \ldots a'_0)}\leq e^{C_{19}\la_3^s},$$
provided both words are admissible.
\end{lemma}

\dem

Let $$f^{i_3}(I^{n-s}_{\alpha}) \subset f^{i_2}(I^{n}_{\beta}) \subset f^{i_1}(I^{n+1}_{\gamma}),$$ with  $1\leq i_3\leq q^{n-s}_\alpha$,  be the intervals corresponding to words
$$(a'_0,\ldots,a'_{n-s-1},a_{n-s},\ldots, a_{n}), (a'_0,\ldots,a'_{n-s-1},a_{n-s},\ldots, a_{n-1}), (a'_0,\ldots,a'_{n-s-1}),$$
respectively. By Lemma \ref{difeocil} there is $j \in \mathbb{Z}$ with  $1\leq i_3+j\leq q_{n-s}^{\pi_1(a_{n-s})},$ such that $$f^{i_3+j}(I^{n-s}_{\alpha}) \subset f^{i_2+j}(I^{n}_{\beta})\subset f^{i_1+j}(I^{n+1}_{\gamma})$$ are the intervals corresponding to words
$$(a''_0,\ldots,a''_{n-s-1},a_{n-s},\ldots,a_{n}), (a''_0,\ldots,a''_{n-s-1},a_{n-s},\ldots, a_{n-1}) \mbox{~and~} (a''_0,\ldots,a''_{n-s-1}),$$ respectively.
Denote

$$\rho_k:=\dfrac{|f^{i_2}(I^{n+1}_{\beta})|\cdot |f^{i_1+k}(I^{n}_{\gamma})|}{|f^{i_1}(I^{n}_{\gamma})|\cdot |f^{i_2+k}(I^{n+1}_{\beta})|},\quad\quad 0\leq k\leq j.$$

We have

\begin{eqnarray*}
\rho_{k+1}&:= &\dfrac{\int_{f^{i_1+k}(I^n_{\gamma})}Df(s)ds}{\int_{f^{i_2+k}(I^{n+1}_{\beta})}Df(s)ds} \cdot\dfrac{|f^{i_2}(I^{n+1}_{\beta})|}{|f^{i_1}(I^n_{\gamma})|}\\
&=&\dfrac{Df(s_{i_1+k}) \ |f^{i_1+k}(I^n_{\gamma})|}{Df(s_{i_2+k})\ |f^{i_2+k}(I^{n+1}_{\beta})|} \cdot\dfrac{|f^{i_2}(I^{n+1}_{\beta})|}{|f^{i_1}(I^n_{\gamma})|}\\
&=&\dfrac{Df(s_{i_1+k})}{Df(s_{i_2+k})}\cdot\rho_k,
\end{eqnarray*}

where $s_{i_1+k}\in f^{i_1+k}(I^n_{\gamma})$ and $s_{i_2+k}\in f^{i_2+k}(I^{n+1}_{\beta}).$ Furthermore,

\begin{eqnarray}
\label{controle}
\exp\{-C_1 \cdot|f^{i_1+k}(I^{n}_{\gamma})|\}\leq \dfrac{Df(s_{i_1+k})}{Df(s_{i_2+k})}\leq \exp\{C_1 \cdot|f^{i_1+k}(I^{n}_{\gamma})|\}.
\end{eqnarray}

Then by \eqref{controle} we have

$$\exp\{-C_1 \sum_{t=0}^{j-1}|f^{i_1+t}(I^{n}_{\gamma})|\}\leq \rho_j\leq \exp\{C_1 \sum_{t=0}^{j-1}|f^{i_1+t}(I^{n}_{\gamma})|\}.$$

However by Proposition \ref{pvz},

$$\sum_{t=0}^{j-1}|f^{i_1+t}(I^{n}_{\gamma})|=\sum_{t=0}^{j-1}|f^{i_3+t}(I^{n-s}_{\alpha})|\cdot\frac{|f^{i_1+t}(I^{n}_{\gamma})|}{|f^{i_3+t}(I^{n-s}_{\alpha})|}\leq C_{20}\cdot\la_3^s,$$
where $C_{20}=C_{20}(C_{17},C_{18},\la)>0$. Taking $C_{19}=C_1\cdot C_{20}$ we obtain the result.

\cqd

%Lemma \ref{markov} shows that the sequence of random variables $\{a_n\}$ can be well enough approximated by the Markov chain.

\begin{lemma}
\label{matriz}
There exists $C_{21}=C_{21}(C_{19},\la)>0$ such that for all $n,m,$
$$e^{-C_{21}}\leq\dfrac{\ell(a_{n+m},\ldots,a_n|a'_{n-1},\ldots,a'_0)}{\ell(a_{n+m},\ldots,a_n|a''_{n-1},\ldots,a''_0)}\leq e^{C_{21}},$$
provided both words $(a'_0,\ldots,a'_{n-1},a_n,\ldots,a_{n+m})$ and $(a''_0,\ldots,a''_{n-1},a_n,\ldots,a_{n+m})$ are admissible.
\end{lemma}

\dem The proof follows easily from Lemma \ref{markov} and the equations

$$\ell(a_{n+m},\ldots,a_n|a'_{n-1},\ldots,a'_0)=\prod_{i=0}^m \ell(a_{n+i}|a_{n+i-1},\ldots, a'_{n-1},\ldots, a'_0),$$

$$\ell(a_{n+m},\ldots,a_n|a''_{n-1},\ldots,a''_0)=\prod_{i=0}^m \ell(a_{n+i}|a_{n+i-1},\ldots, a''_{n-1},\ldots, a''_0).$$

\cqd

\begin{lemma}
\label{juncao}\label{juncao1}
Let $a_k\ldots a_n$, $k\leq n$,  and $a'_n a'_{n+1}\ldots a'_{n+m}$ be  two admissible words, such that $a_n=(\alpha, \chi, n)$, $a'_n=(\alpha,\chi',n)$.  Then $$a_k \ldots a_n  a'_{n+1} \ldots a'_{n+m}$$ is admissible.
\end{lemma}

\dem  Since   $a_k\ldots a_n$ is admissible then there exists $a_0, a_1, \dots, a_{k-1}$ such that $a_0\dots a_{k-1}a_k \dots a_n$ is admissible.  If       $a'_n a'_{n+1}\ldots a'_{n+m}$ is admissible then there exists an admissible word with the form    $a'_0 \ldots a'_n a'_{n+1}\ldots a'_{n+m}$. Note that the functions $k_1(x)=k_1'$, $k_2(x)=k_2'$, $\dots$, $k_{n+m}(x)=k_{n+m}'$ are constant in $[a'_0 \ldots a'_n a'_{n+1}\ldots a'_{n+m}]$  and
$$f^{k'_n}[a'_0 \ldots a'_n a'_{n+1}\ldots a'_{n+m}]\subset f_n(I_\alpha^n).$$
The functions $k_1(x)=k_1$, $k_2(x)=k_2$, $\dots$, $k_n(x)=k_n$ are constant in $[a_0\ldots a_n]$  and
$$f^{k_n}[a_0\ldots a_n]=f_n(I_\alpha^n).$$
In particular every $x$ in the non empty  set
$$J=f^{-k_n}f^{k'_n}[a'_0 \ldots a'_n a'_{n+1}\ldots a'_{n+m}]\cap \Lambda\subset [a_0\ldots a_n]$$
belongs to the cylinder $[a_0 \ldots a_n  a'_{n+1} \ldots a'_{n+m}]$. Indeed, since $x \in [a_0\ldots a_n]$ we have $s_i(x)=a_i$, $k_i(x)=k_i$,  for $0\leq i \leq n$. Note that $f^{k_n}(x)=f^{k_n'}(y)$, for some $y \in [a'_0 \ldots a'_n a'_{n+1}\ldots a'_{n+m}]\cap \Lambda$. Then $$k_{i}(x)=k_{i}(y)-k_n'+k_n=k_i'-k_n'+k_n$$ for $n \leq i \leq n+m$, since
$$f^{k_i'-k_n'+k_n}(x)=f^{k_i'}(y)\in I^i,$$
and moreover if $k_n\leq j < k_i'-k_n'+k_n$ then
$$f^{j}(x)=f^{j-k_n+k_n'}(y)\not\in I^i,$$
since $ 0\leq j-k_n +k_n' < k_i'=k_i(y)$ and if $j < k_n$ then
$$f^{j}(x) \not\in I^i,$$ because
$j < k_n\leq k_i(x)$. This implies that $s_i(x)=a'_i$ for $n< i\leq  n+m$.
 \cqd

\begin{lemma}
\label{completar}
Let $\alpha,\beta\in\A.$ For each $n$  there is an  admissible word $a_n \ldots a_{n+k}$  with $\pi_1(a_n)=\beta$, $\pi_1(a_{n+k})=\alpha$.
\end{lemma}

\dem Firstly we claim that if $\alpha\in \mathcal{A}$ does not lose in the nth level then  there  is an admissible word $b_n b_{n+1}$ such that
$\pi_1(b_n)=\pi_1(b_{n+1})=\alpha$. Indeed in this case $f_{n+1}(I^{n+1}_\alpha)\subset f_n(I^n_\alpha)$. This implies that the word $(\alpha, 0, n)(\alpha,0,n+1)$ is admissible.

Second, if $\alpha$ loses and $\beta \in \mathcal{A}$ wins in the nth level then  there  is an admissible word $b_n b_{n+1}$ such that
$\beta= \pi_1(b_n)$ and $\alpha=\pi_1(b_{n+1})$ and an admissible word $b'_n b'_{n+1}$ such that
$\alpha= \pi_1(b'_n)=\pi_1(b'_{n+1})$. Indeed if $f_n$ has type $1$ then we have that  $I^n_\alpha \subset f_n(I^n_\beta)$, $I^n_\alpha$ is not inside $I^{n+1}$ and enters $I^{n+1}$ after one iteration of $f_n$, landing in $f_n(I^n_\alpha)=f_{n+1}(I^{n+1}_\alpha)$,  so the word $(\beta,0,n)(\alpha,1,n+1)$ is admissible. Note also that $f_n(I^n_\alpha)=f_{n+1}(I^{n+1}_\alpha)$ so $(\alpha, 0, n)(\alpha,0,n+1)$ is admissible.  if $f_n$ has type $0$ then we have that
$f_{n+1}(I^{n+1}_\alpha)\subset f_n(I^n_\beta)$, so the word $(\beta,0,n)(\alpha,0,n+1)$ is admissible. Furthermore $f_n(I^n_\alpha)$ is not inside $I^{n+1}$ and it enters $I^{n+1}$ after one iteration of $f_n$, landing in $f_{n+1}(I^{n+1}_\alpha)$, so the word $(\alpha,0,n)(\alpha,1,n+1)$ is admissible.

In particular, using Lemma \ref{juncao} it follows that for every $m\geq 0$, $p> 0$ and $\alpha \in \mathcal{A}$  there exists a word $\omega=a_ma_{m+1}\dots a_{m+p}$ such that $\pi_1(a_i)=\alpha$ for every $m\leq i\leq m+p$.

Now suppose that $\beta$ wins from $\alpha$ in the $(m-1)$-th renormalization. Then as we saw above $(\beta,0,m-1)(\alpha,\epsilon_{m-1},m)$ and $\omega$ are  admissible. By Lemma \ref{juncao} there exists a word $(\beta,0,m-1)a'_ma_{m+1}\dots a_{m+p}$ such that $\pi_1(a_{m+p})=\alpha$.

Finally, since $f \in \mathcal{B}_k$, given $\alpha$, $\beta \in \mathcal{A}$, there exists a sequence of letters $\alpha_i$, and levels $n_i$, $i \leq s$, $n\leq n_i < n_{i+1}\leq n+k$ for every  $i < s$, such that $\alpha_0=\beta$, $\alpha_s=\alpha$ and $\alpha_i$ wins from $\alpha_{i+1}$ in the $n_i$th level. So there are admissible words
$a^i_{n_i}\dots a^i_{n_{i+1}}$ such that $\pi_1(a^i_{n_i})=\alpha_i$ and $\pi_1(a^i_{n_{i+1}})=\alpha_{i+1}$. By Lemma \ref{juncao} there is an admissible word $a_{n_0}\dots a_{n_s}$ such that $\pi_1(a_{n_0})=\beta$ and $\pi_1(a_{n_s})=\alpha$.

Since we already proved that there exist admissible words $b_n\dots b_{n_0}$ and $c_{n_s}\dots c_{n+k}$ such that $\pi_1(b_{n_0})=\alpha$, $\pi_1(b_{n_s})=\alpha$,  $\pi_1(c_{n_s})=\beta$, $\pi_1(c_{n+k})=\beta$, by Lemma \ref{juncao} again it  there exists a word of type
$$b'_n\dots b_{n_0}a_{n_0+1}\dots a_{n_s-1} c_{n_s}\dots c'_{n+k}$$
with $\pi_1(b'_n)=\alpha$ and $\pi_1(c'_{n+k})=\beta$.
\cqd

The proof of the following lemma is simple:

\begin{lemma}
\label{mp}
There exists $C_{22}>0$ such that for all $n,m,$ and all  admissible words $a'_0\ldots a'_{n-k},$ $a''_0 \ldots a''_{n-k},$ $a_n\ldots a_{n+m}$
 $$e^{-C_{22}}\leq\dfrac{\ell(a_{n+m} \ldots a_n|a'_{n-k} \ldots a'_0)}{\ell(a_{n+m} \ldots a_n|a''_{n-k} \ldots a''_0)}\leq e^{C_{22}}.$$
\end{lemma}

\begin{pro}
\label{emc}There are  $C_{23}>0$ and $0<\la_4<1$ such that
$$| \ell(a_n|a_{n-r} \ldots a_0)- \ell(a_n)|\leq C_{23}\cdot \la_4^{\sqrt{r}},$$
where $r=[\frac{n}{2}].$
\end{pro}
\dem
Indeed Proposition \ref{emc} is a Markov ergodic theorem and it can be proved by the methods of the theory Markov chains, as in Khanin and Sinai \cite{newproof} (see also Sinai \cite{sinaibook}). Thus, we shall describe only the main steps. Fix an integer $m,$ $m\thicksim \sqrt{\frac{n}{2}}$ and introduce a new probability measure on the words of the form
$$
\begin{array}{cl}
\tilde{a}=&(a_n a_{n-1} \ldots a_{n-m+k}  \ a_{n-m} \ldots  a_{n-2m+k} \\
& a_{n-2m} \ldots  a_{n-3m+k} \ldots a_{n-(i-1)m} \ldots  a_{n-im+k} \  a_{n-im} \ldots a_0)
\end{array}
$$
by the formula
$$
\begin{array}{lll}
\ell'(\tilde{a})&=&\ell(a_0 \ldots a_{n-im})\ell(a_{n-(i-1)m} \ldots  a_{n-im+3}|a_{n-im} \ldots a_0)\\
& &\displaystyle\prod_{s=0}^{i-2}\ell(a_{n-sm} \ldots  a_{n-(s+1)m+k}|a_{n-(s+1)m} \ldots a_{n-(s+2)m+k}).
\end{array}
$$
Here  $i\thicksim \sqrt{\frac{n}{2}}.$  It follows easily from Lemma \ref{markov} that
\begin{equation} \label{asseq} \exp(-C_{19}\cdot\la_3^m\cdot i)\leq\dfrac{\ell'(\tilde{a})}{\ell(\tilde{a})}\leq\exp(C_{19}\cdot\la_3^m\cdot i).\end{equation}
The Lemma \ref{mp} shows that the Markov transition operator corresponding to $\ell'$ for the transition to $m$ steps is a contraction for the apropriate Cayley-Hilbert metric, and this contraction is {\it uniformly } smaller  than $1$ on each step. Then the usual Ergodic Theorem for Markov chains shows that the conditional probabilities $\ell'(a_n|a_{n-im} \ldots a_0)$ assymptotically do not depend on $a_{n-im} \ldots a_0$. Due (\ref{asseq}) the same holds for $\ell(a_n|a_{n-im} \ldots a_0)$. This gives the desired result.
\cqd

Denote by $\ell(\alpha, \star, n)$ the Lebesgue measure of the union $[(\alpha,0,n)]\cup [(\alpha,1,n)]$. Note that
$$\ell(\alpha, \star, n) =\frac{\sum_{i=1}^{q^n_\alpha} |f^i(I^n_\alpha)|}{|I|}.$$

%\begin{pro}
%\label{principal} There is  $0<\la_5=\la_5(k,\la)<1$ such that for all $\al\in\mathcal{A}$ and $n \in \mathbb{N}$
%$$\int_{I^n_{\al}}\dfrac{f_n''(s)}{f_n'(s)}ds=\ell((\alpha,\star,n))\cdot\int_{I}\dfrac{f''(s)}{f'(s)}ds+O(\la_5^{\sqrt{\frac{n}{2}}}).$$
%\end{pro}
\begin{proof}[Proof of Theorem \ref{principal}.] For simplify the notation we use $f_n$ to denote $R^n(f).$ Let $r=[\frac{n}{2}].$ We rewrite $\int_{I^n_{\al}}\frac{D^2f_n(s)}{Df_n(s)}ds$ in the following way. By the mean value theorem for integrals
\begin{eqnarray*}
\int_{I^n_{\al}}\dfrac{D^2f_n(s)}{Df_n(s)}ds&=& \sum_{\beta\in\mathcal{A}}\sum_{j=1}^{q^{\beta}_r}\,\sum_{f^i(I^n_{\al})\subset f^j(I^r_{\beta})}\int_{f^i(I^n_{\al})}\dfrac{D^2f(s)}{Df(s)}ds\\
 &= &\sum_{\beta\in\mathcal{A}}\sum_{j=1}^{q^{\beta}_r}\,\sum_{f^i(I^n_{\al})\subset f^j(I^r_{\beta})} \dfrac{D^2f(x_j^\alpha)}{Df(x_j^\alpha)}\cdot|f^i(I^n_{\al})| ,
\end{eqnarray*}
where $x_i^\alpha\in f^i(I^n_{\alpha})$. In a similar way we can choose   $y_j^\beta \in f^j(I^r_\beta) $ such that
\begin{eqnarray*}
\int_{I}\dfrac{D^2f(s)}{Df(s)}ds= \sum_{\beta\in\A}\sum_{j=1}^{q^{\beta}_r}\dfrac{D^2f(y_j^{\beta})}{Df(y_j^{\beta})}\cdot|f^j(I^r_{\beta})|.
\end{eqnarray*}
So
\begin{eqnarray*}
\int_{I^n_{\al}}\dfrac{D^2f_n(s)}{Df_n(s)}ds&=&\sum_{\beta\in\mathcal{A}}\sum_{j=1}^{q^{\beta}_r}\,\sum_{f^i(I^n_{\al})\subset f^j(I^r_{\beta})}\left(\dfrac{D^2f(x_j)}{Df(x_j)}-\dfrac{D^2f(y_j^{\beta})}{Df(y_j^{\beta})}\right)\cdot|f^i(I^n_{\al})| \\
& +& \sum_{\beta\in\mathcal{A}}\sum_{j=1}^{q^{\beta}_r}\,\sum_{f^i(I^n_{\al})\subset f^j(I^r_{\beta})} \dfrac{D^2f(y_j^{\beta})}{Df(y_j^{\beta})}\cdot|f^i(I^n_{\al})|,
\end{eqnarray*}

Due to the smooth properties of $f$ the first term is at most  $C_{24}\cdot\la_6^{\frac{n}{2}\nu},$ where $C_{24}=C_{24}(\la,k)>0$ and $0<\la_6=\la_6(\la,k)<1$.We will now analyze the second term.

%The second parcel can be rewritten as

%$$\sum_{\beta\in\mathcal{A}}\sum_{j=0}^{q^{\beta}_r-1} \dfrac{f''(y_j^{\beta})}{f'(y_j^{\beta})}|f^{j}(I^r_{\beta})|p_j^r,$$
%where
%$$p_j^r=\dfrac{1}{|f^{j}(I^r_{\beta})}\sum_{i:f^i(I^n_{\al})\subset f^j(I^r_{\beta})}|f^i(I^n_{\al})|.$$

%Assume that we have succeeded in proving that for some constant $p$ the difference

%\begin{eqnarray}
%\label{p}
%|p_j^r-p|\leq \const\cdot \la_1^{\sqrt{n}}.
%\end{eqnarray}

\begin{eqnarray*}
\lefteqn{\sum_{\beta\in\mathcal{A}}\sum_{j=1}^{q^{\beta}_r}\,\sum_{f^i(I^n_{\al})\subset f^j(I^r_{\beta})} \dfrac{D^2f(y_j^{\beta})}{Df(y_j^{\beta})}\cdot|f^i(I^n_{\al})|=}\\
%& &=\sum_{\beta\in\mathcal{A}}\sum_{j=0}^{q^{\beta}_r-1}\,\sum_{i:f^i(I^n_{\al})\subset f^j(I^r_{\beta})} \dfrac{f''(y_j^{\beta})}{f'(y_j^{\beta})}\cdot |f^j(I^r_{\beta})|\cdot\dfrac{|f^i(I^n_{\al})|}{|f^j(I^r_{\beta})|}\\
 & &= \sum_{\beta\in\mathcal{A}}\sum_{j=1}^{q^{\beta}_r}\frac{D^2f(y_j^{\beta})}{Df(y_j^{\beta})}\cdot |f^j(I^r_{\beta})|\dfrac{\sum_{f^i(I^n_{\al})\subset f^j(I^r_{\beta})} |f^i(I^n_{\al})|}{|f^j(I^r_{\beta})|}\\
%&=& \sum_{\beta\in\mathcal{A}}\sum_{j=0}^{q^{\beta}_r-1}\dfrac{f''(y_j^{\beta})}{f'(y_j^{\beta})}\cdot |f^j(I^r_{\beta})| \cdot l\left((\al,*)\,|\,[f^j(I^r_{\beta})]\right)\\
 %& & \pm  \sum_{\beta\in\mathcal{A}}\sum_{j=0}^{q^{\beta}_r-1}\dfrac{f''(y_j^{\beta})}{f'(y_j^{\beta})}\cdot |f^j(I^r_{\beta})| \cdot l\left((\al,0)\right)\\
 & & =  \sum_{\beta\in\mathcal{A}}\sum_{j=1}^{q^{\beta}_r}\dfrac{D^2f(y_j^{\beta})}{Df(y_j^{\beta})}\cdot |f^j(I^r_{\beta})| \cdot \left[\ell\Big((\al,\star,n)\,|\,[f^j(I^r_{\beta})]\Big)-\ell\Big((\alpha,\star,n)\Big)\right]\\
 & & +  \ell\Big((\alpha,\star,n)\Big)\cdot\sum_{\beta\in\mathcal{A}}\sum_{j=1}^{q^{\beta}_r}\dfrac{D^2f(y_j^{\beta})}{Df(y_j^{\beta})}\cdot |f^j(I^r_{\beta})|\\
 & & = (\mathrm{IV})+(\mathrm{V}).
\end{eqnarray*}
By Proposition \ref{emc} we have that $(\mathrm{IV})=\mathrm{O}(\la_4^{\sqrt{\frac{n}{2}}}).$ Now observe that $(\mathrm{V})$ is a Riemann sum for the integral  $\int_{I}\dfrac{D^2f(s)}{Df(s)}ds$. By
Proposition~\ref{pvz} and  $\nu$-Holder continuity of $\dfrac{D^2f}{Df}$ we have
$$(\mathrm{V}) = \ell\Big((\alpha,\star,n)\Big)\cdot \int_{I}\dfrac{D^2f(s)}{Df(s)}ds+ O(\lambda^{n/k}).$$   This finishes the proof.
\end{proof}

Before proving the Theorem \ref{teo2} we need the following lemma whose proof is left to the reader.

\begin{lemma}
\label{M}
Let $a,b\in\mathbb{R}.$ Then for every $C> 0$ there is $C_{25}>0$ such that if $|a|,|b| \leq C$ then
$$|M_{a}-M_{b}|_{C^2}\leq C_{25} \cdot |a-b|,$$
where $M_a$ and $M_b$  are defined in \eqref{defMoebius}.
\end{lemma}

\begin{proof}[Proof of Theorem \ref{teo2}.] By assumption $\int_{I}\dfrac{D^2f(s)}{Df(s)}ds=0$, so  by Theorem \ref{principal} we have

$$\left|\int_{I^n_\al}\dfrac{D^2f_n(s)}{Df_n(s)}ds\right|\leq C_{26}\cdot \la_4^{\sqrt{\frac{n}{2}}}.$$
Therefore by Lemma \ref{M}

\begin{eqnarray}
\label{afim}
\left|M_{\int_{I^n_\al}\frac{D^2f_n(s)}{Df_n(s)}ds}-\mathrm{Id}\right|_{C^2}\leq C_{25}\cdot \left|\int_{I^n_\al}\dfrac{D^2f_n(s)}{Df_n(s)}ds-0 \right|
\end{eqnarray}
$$
\leq  C_{25}\cdot C_{26}\cdot \la_4^{\sqrt{\frac{n}{2}}}.
$$

Theorem \ref{teo1} together with \eqref{afim} gives us that
$$ \left|\mathcal{Z}_{I^n_{\al}}(R^n(f))-\mathrm{Id}\right|_{C^2}\leq C_{27}\cdot \la_4^{\sqrt{\frac{n}{2}}}.$$
\end{proof}

\bibliographystyle{elsarticle-num}



\end{document}